\documentclass[12pt,reqno]{article}

\usepackage[usenames]{color}
\usepackage{amssymb}
\usepackage{graphicx}
\usepackage{amscd}

\usepackage{amsthm}
\newtheorem{theorem}{Theorem}

\newtheorem{proposition}[theorem]{Proposition}
\newtheorem{corollary}[theorem]{Corollary}
\newtheorem{conjecture}[theorem]{Conjecture}

\theoremstyle{definition}

\newtheorem{example}[theorem]{Example}

\usepackage[colorlinks=true,
linkcolor=webgreen, filecolor=webbrown,
citecolor=webgreen]{hyperref}

\definecolor{webgreen}{rgb}{0,.5,0}
\definecolor{webbrown}{rgb}{.6,0,0}

\usepackage{color}
\usepackage{float}
\usepackage{graphics,amsmath,amssymb}
\usepackage{amsfonts}
\usepackage{latexsym}
\usepackage{epsf}

\DeclareMathOperator{\Rev}{Rev}
\DeclareMathOperator{\rev}{rev}

\newcommand{\seqnum}[1]{\href{http://oeis.org/#1}{\underline{#1}}}

\usepackage{tikz}
\newcounter{x}
\newcounter{y}
\newcounter{z}

\newcommand\xaxis{210}
\newcommand\yaxis{-30}
\newcommand\zaxis{90}

\newcommand\topside[3]{
  \fill[fill=yellow, draw=black,shift={(\xaxis:#1)},shift={(\yaxis:#2)},
  shift={(\zaxis:#3)}] (0,0) -- (30:1) -- (0,1) --(150:1)--(0,0);
}

\newcommand\leftside[3]{
  \fill[fill=red, draw=black,shift={(\xaxis:#1)},shift={(\yaxis:#2)},
  shift={(\zaxis:#3)}] (0,0) -- (0,-1) -- (210:1) --(150:1)--(0,0);
}

\newcommand\rightside[3]{
  \fill[fill=blue, draw=black,shift={(\xaxis:#1)},shift={(\yaxis:#2)},
  shift={(\zaxis:#3)}] (0,0) -- (30:1) -- (-30:1) --(0,-1)--(0,0);
}

\newcommand\cube[3]{
  \topside{#1}{#2}{#3} \leftside{#1}{#2}{#3} \rightside{#1}{#2}{#3}
}

\newcommand\planepartition[1]{
 \setcounter{x}{-1}
  \foreach \a in {#1} {
    \addtocounter{x}{1}
    \setcounter{y}{-1}
    \foreach \b in \a {
      \addtocounter{y}{1}
      \setcounter{z}{-1}
      \foreach \c in {1,...,\b} {
        \addtocounter{z}{1}
        \cube{\value{x}}{\value{y}}{\value{z}}
      }
    }
  }
}
\begin{document}

\begin{center}
\vskip 1cm{\LARGE\bf Centered polygon numbers, heptagons and nonagons, and the Robbins numbers} \vskip 1cm \large
Paul Barry\\
School of Science\\
Waterford Institute of Technology\\
Ireland\\
\href{mailto:pbarry@wit.ie}{\tt pbarry@wit.ie}
\end{center}
\vskip .2 in

\begin{abstract} In this note, we explore certain determinantal descriptions of the Robbins numbers. Techniques used for this include continued fractions, Riordan arrays and series inversion. Proven and conjectured representations involve the determinants of both Hankel and symmetric matrices. In specific cases, links are drawn to centered polygonal numbers, and to heptagons and nonagons. We conjecture a Hankel transform determinant for the Robbins numbers related to the Fibonacci and the Catalan numbers.\end{abstract}

\section{Introduction} The Robbins numbers $A_n$ \seqnum{A005130}, which begin
$$1,1, 2, 7,42, 429, 7436, \cdots$$ can be defined by
$$A_n=\prod_{k=0}^{n-1} \frac{(3k+1)!}{(n+k)!}.$$ They count $n \times n$ alternating sign matrices. An alternating sign matrix (ASM) is a matrix with entries drawn from the set $\{-1,0,1\}$,
such that $1$’s and $-1$'s
alternate in each column and each row (when a $-1$ occurs), and such that the first and last non-zero entry in each row and column is $1$. They extend the partially ordered set of permutation matrices into a lattice. For instance, the $7$ $3 \times 3$ alternating sign matrices are as follows.
$$\left(
\begin{array}{ccc}
 1 & 0 & 0 \\
 0 & 1 & 0 \\
 0 & 0 & 1 \\
\end{array}
\right), \left(
\begin{array}{ccc}
 1 & 0 & 0 \\
 0 & 0 & 1 \\
 0 & 1 & 0 \\
\end{array}
\right),\left(
\begin{array}{ccc}
 0 & 1 & 0 \\
 1 & 0 & 0 \\
 0 & 0 & 1 \\
\end{array}
\right),\left(
\begin{array}{ccc}
 0 & 1 & 0 \\
 1 & -1 & 1 \\
 0 & 1 & 0 \\
\end{array}
\right),$$
$$\left(
\begin{array}{ccc}
 0 & 1 & 0 \\
 0 & 0 & 1 \\
 1 & 0 & 0 \\
\end{array}
\right), \left(
\begin{array}{ccc}
 0 & 0 & 1 \\
 1 & 0 & 0 \\
 0 & 1 & 0 \\
\end{array}
\right), \left(
\begin{array}{ccc}
 0 & 0 & 1 \\
 0 & 1 & 0 \\
 1 & 0 & 0 \\
\end{array}
\right).$$
Every permutation matrix is an alternating sign matrix, and thus there are $A_n-n!$  (OEIS sequence \seqnum{A321511}) alternating sign matrices which have at least one $-1$ among their elements.

A plane partition is a two-dimensional array of nonnegative integers $\pi_{i,j}$  (with positive integer indices $i$ and $j$) that is non-increasing in both indices.

In this note we explore links between the Robbins numbers, Hankel transforms, principal minor sequences of symmetric matrices, and Riordan arrays. The plan of the paper is as follows.

\begin{enumerate}
\item This Introduction
\item Transforms, Riordan arrays, continued fractions, and Hankel transforms
\item Centered polygon numbers
\item Heptagons and nonagons
\item A special matrix
\item The sequence $1,-2,-7,42,429,\ldots$ as a Hankel transform
\item The sequence $1,1,2,6,33,286,\ldots$ as a Hankel transform
\item The sequence $1,3,26,646,\ldots$ as a Hankel transform
\item The sequence $1,2,11,170,7429,\ldots$ as a Hankel transform
\item Fibonacci numbers, Catalan numbers, and the Robbins numbers
\item Riordan arrays and the Robbins numbers
\item Conclusion
\item Acknowledgements.
\end{enumerate}

Underlying the simple formula
  $$A_n=\prod_{k=0}^{n-1} \frac{(3k+1)!}{(n+k)!}$$ is a rich history of research, that linked the enumeration of classes of alternating sign matrices to that of certain plane partitions, via an excursion into mathematical physics in the guise of the six-vertex model \cite{Hone, Ilse, Izergin, Korepin_book}. Indeed, from this field, we have the following simple determinantal expression \cite{BDZ} for the Robbins numbers.

We consider the (semi-infinite) matrix $M$ given by
$$(\binom{n+k}{k})- (\text{If}[k=n+1, 1,0]),$$ which begins
$$\left(
\begin{array}{ccccccc}
 1 & 1 & 1 & 1 & 1 & 1 & 1 \\
 1 & 2 & 3 & 4 & 5 & 6 & 7 \\
 1 & 3 & 6 & 10 & 15 & 21 & 28 \\
 1 & 4 & 10 & 20 & 35 & 56 & 84 \\
 1 & 5 & 15 & 35 & 70 & 126 & 210 \\
 1 & 6 & 21 & 56 & 126 & 252 & 462 \\
 1 & 7 & 28 & 84 & 210 & 462 & 924 \\
\end{array}
\right)-\left(
\begin{array}{ccccccc}
 0 & 1 & 0 & 0 & 0 & 0 & 0 \\
 0 & 0 & 1 & 0 & 0 & 0 & 0 \\
 0 & 0 & 0 & 1 & 0 & 0 & 0 \\
 0 & 0 & 0 & 0 & 1 & 0 & 0 \\
 0 & 0 & 0 & 0 & 0 & 1 & 0 \\
 0 & 0 & 0 & 0 & 0 & 0 & 1 \\
 0 & 0 & 0 & 0 & 0 & 0 & 0 \\
\end{array}
\right)$$  or
$$\left(
\begin{array}{ccccccc}
 1 & 0 & 1 & 1 & 1 & 1 & 1 \\
 1 & 2 & 2 & 4 & 5 & 6 & 7 \\
 1 & 3 & 6 & 9 & 15 & 21 & 28 \\
 1 & 4 & 10 & 20 & 34 & 56 & 84 \\
 1 & 5 & 15 & 35 & 70 & 125 & 210 \\
 1 & 6 & 21 & 56 & 126 & 252 & 461 \\
 1 & 7 & 28 & 84 & 210 & 462 & 924 \\
\end{array}
\right).$$
The generating function of this matrix is given by
$$f(x,y)=\frac{1}{1-x-y}-\frac{y}{1-xy}.$$
The sequence of principal minors $|M|_{0\le i,j \le n}$ yields the Robbins numbers $A_{n+1}$
$$1, 2, 7,42, 429, 7436, \cdots.$$
We now multiply the matrix $M$ on the right by the Riordan array $\left(\frac{1}{(1-x)\left(\frac{1}{1-x}-x\right)},x\right)$ to obtain the symmetric matrix that begins
$$\left(
\begin{array}{ccccccc}
 1 & 1 & 1 & 1 & 1 & 1 & 1 \\
 1 & 3 & 4 & 5 & 6 & 7 & 8 \\
 1 & 4 & 9 & 14 & 20 & 27 & 35 \\
 1 & 5 & 14 & 29 & 49 & 76 & 111 \\
 1 & 6 & 20 & 49 & 99 & 175 & 286 \\
 1 & 7 & 27 & 76 & 175 & 351 & 637 \\
 1 & 8 & 35 & 111 & 286 & 637 & 1275 \\
\end{array}
\right).$$
Multiplying on the right is the same as the following application of the fundamental theorem of Riordan arrays (in $y$):
$$\left(\frac{1}{(1-y)\left(\frac{1}{1-y}-y\right)},y\right)\cdot f(x,y)=\frac{1}{(1-y)\left(\frac{1}{1-y}-y\right)}f(x,y)=\frac{1}{(1-xy)(1-x-y)}.$$ The form of this generating function shows that the new matrix is symmetrical.
As the Riordan array used is lower triangular with $1$'s on the diagonal, we see that the principal minors of the new symmetric matrix $\tilde{M}$ once again yield the Robbins numbers $A_{n+1}$. We now consider the (reversed) embedded matrix
$$\left(
\begin{array}{ccccccc}
 1 & 0 & 0 & 0 & 0 & 0 & 0 \\
 3 & 1 & 0 & 0 & 0 & 0 & 0 \\
 9 & 4 & 1 & 0 & 0 & 0 & 0 \\
 29 & 14 & 5 & 1 & 0 & 0 & 0 \\
 99 & 49 & 20 & 6 & 1 & 0 & 0 \\
 351 & 175 & 76 & 27 & 7 & 1 & 0 \\
 1275 & 637 & 286 & 111 & 35 & 8 & 1 \\
\end{array}
\right).$$

This is in fact a Riordan matrix, namely
\begin{align*}(1-3x+3x^2-2x^3, x(1-x))^{-1}&=\left(\frac{1}{1-3xc(x)+3x^2c(x)^2-2x^3c(x)^3}, xc(x)\right)\\&=\left(\frac{1}{(1-x)\sqrt{1-4x}}, xc(x)\right),\end{align*}
where $c(x)=\frac{1-\sqrt{1-4x}}{2x}$ is the generating function of the Catalan numbers $C_n=\frac{1}{2n+1}\binom{2n}{n}$. Denoting this Riordan array by $R$, we then have that
$$\tilde{M}=\text{reversal}(R)+(\text{reversal}(R)-\text{diag}R_0)^T,$$
$$\tilde{M}_{n,k}=\begin{cases} R_{n,n-k}\quad \text{if\,}k \le n,\\
                        R_{k,k-n} \quad \text{otherwise},
           \end{cases}$$ where
$R_0$ is the initial column of $R$, and the reversal of $R$ is given by the lower-triangular part of $\tilde{M}$ in this case. We shall call this process of going from a Riordan array to a symmetric matrix the \emph{Riordan symmetrization process}. (More generally, the process can be applied to any lower-triangular matrix to produce a symmetric matrix).
We thus have arrived at the following proposition.
\begin{proposition} The Robbins numbers $A_{n+1}$ are given by the principal minors of the matrix obtained by the Riordan symmetrization process starting with the Riordan array
$$\left(\frac{1}{(1-x)\sqrt{1-4x}}, xc(x)\right).$$
\end{proposition}
The matrix defined by $\left(\frac{1}{(1-x)\sqrt{1-4x}}, xc(x)\right)$ has general $(n,k)$-th element
$$\sum_{j=0}^n \binom{2j-k}{j-k}.$$
This gives us the following result.
\begin{corollary}
We have
$$A_{n+1}= \left| \sum_{j=0}^n \binom{2j-n+k}{j-n+k} \right|_{0 \le i,j \le n}.$$
\end{corollary}
We now remark that the Riordan array $\left(\frac{1}{(1-x)\sqrt{1-4x}}, xc(x)\right)$ is the vertical half \cite{Vertical} of the Riordan array  $\left(\frac{1}{(1-x)(1-x+x^2)},\frac{x}{1-x}\right)$. Thus we have the following result.
\begin{corollary} The Robbins numbers $A_{n+1}$ are given by the principal minors of the matrix obtained by the Riordan symmetrization process applied to the vertical half of the Riordan array $\left(\frac{1}{(1-x)(1-x+x^2)},\frac{x}{1-x}\right)$.
\end{corollary}
\begin{example} \textbf{Diagonal sums.} The diagonal sums of the matrix with generating function $\frac{1}{(1-xy)(1-x-y)}$ will have generating function $\frac{1}{(1-x^2)(1-2x)}$. This corresponds to the sequence \seqnum{A000975}, which begins $$1, 2, 5, 10, 21, 42, 85, 170, 341, 682, 1365,\ldots.$$
This is the convolution of the all $1$'s sequence and the Jacobsthal numbers \seqnum{A001045}. The Jacobsthal numbers play a distinguished role in the story of the Robbins numbers:  $A_n$ is odd if and only if $n$ is a Jacobsthal number \cite{Frey}. They also occur in other contexts related to the Robbins numbers \cite{Jac}.
\end{example}
These results indicate that it is worthwhile exploring more links between the construction of the Robbins numbers and Riordan arrays. The next example reinforces this.
\begin{example} We consider the Riordan array $\left(\frac{1}{1+x+x^2}, \frac{x}{1+x}\right)$ which begins
$$\left(
\begin{array}{ccccccc}
 1 & 0 & 0 & 0 & 0 & 0 & 0 \\
 -1 & 1 & 0 & 0 & 0 & 0 & 0 \\
 0 & -2 & 1 & 0 & 0 & 0 & 0 \\
 1 & 2 & -3 & 1 & 0 & 0 & 0 \\
 -1 & -1 & 5 & -4 & 1 & 0 & 0 \\
 0 & 0 & -6 & 9 & -5 & 1 & 0 \\
 1 & 0 & 6 & -15 & 14 & -6 & 1 \\
\end{array}
\right).$$
The bivariate generating function of this matrix is given by
$$ f(x,y)=\frac{\frac{1}{1+x+x^2}}{1-y\frac{x}{1+x}}=\frac{1+x}{(1+x+x^2)(1+x-xy)}.$$
The symmetrization of this matrix begins
$$\left(
\begin{array}{ccccccc}
 1 & 1 & 1 & 1 & 1 & 1 & 1 \\
 1 & -1 & -2 & -3 & -4 & -5 & -6 \\
 1 & -2 & 0 & 2 & 5 & 9 & 14 \\
 1 & -3 & 2 & 1 & -1 & -6 & -15 \\
 1 & -4 & 5 & -1 & -1 & 0 & 6 \\
 1 & -5 & 9 & -6 & 0 & 0 & 0 \\
 1 & -6 & 14 & -15 & 6 & 0 & 1 \\
\end{array}
\right).$$
The principal minors of this matrix then begin
$$1, -2, -7, 42, 429, -7436, -218348,\ldots.$$
The generating function of the symmetrization is given by
$$f\left(xy, \frac{1}{x}\right)+f\left(xy, \frac{1}{y}\right)-\frac{1}{1+xy+x^2y^2}=\frac{1}{(1-x+xy)(1-y+xy)}.$$
We can also consider the symmetric matrix with generating function
$$\frac{1}{(1+x+xy)(1+y+xy)}.$$
This matrix begins
$$\left(
\begin{array}{ccccccc}
 1 & -1 & 1 & -1 & 1 & -1 & 1 \\
 -1 & -1 & 2 & -3 & 4 & -5 & 6 \\
 1 & 2 & 0 & -2 & 5 & -9 & 14 \\
 -1 & -3 & -2 & 1 & 1 & -6 & 15 \\
 1 & 4 & 5 & 1 & -1 & 0 & 6 \\
 -1 & -5 & -9 & -6 & 0 & 0 & 0 \\
 1 & 6 & 14 & 15 & 6 & 0 & 1 \\
\end{array}
\right).$$ Again, its principal minors yield the sequence
$$1, -2, -7, 42, 429, -7436, -218348,\ldots.$$
In both cases, if we multiply the $(n,k)$-th element of the matrix by $(-1)^k$ and take the principal minors, then we get the unsigned sequence
$$1, 2, 7, 42, 429, 7436, 218348,\ldots.$$
We can also consider the matrix with generating function
$$\frac{1}{(1-ix+xy)(1-iy+xy)},$$ which expands to give the complex matrix that begins
$$\left(
\begin{array}{ccccccc}
 1 & i & -1 & -i & 1 & i & -1 \\
 i & 1 & 2 i & -3 & -4 i & 5 & 6 i \\
 -1 & 2 i & 0 & 2 i & -5 & -9 i & 14 \\
 -i & -3 & 2 i & -1 & i & -6 & -15 i \\
 1 & -4 i & -5 & i & -1 & 0 & -6 \\
 i & 5 & -9 i & -6 & 0 & 0 & 0 \\
 -1 & 6 i & 14 & -15 i & -6 & 0 & 1 \\
\end{array}
\right).$$
The principal minor sequence of this matrix will then be
$$1, 2, 7, 42, 429, 7436, 218348,\ldots.$$
\end{example}
\begin{example} $\mathbf{2}-$\textbf{factorial numbers}. The generating function $\frac{1}{(1-x-y-xy)(1-x-y)}$ has an interesting application. Using the variant $1+\frac{xy(1-x)(1-y)}{(1-x-y-xy)(1-x-y)}$ we obtain a matrix that begins
$$\left(
\begin{array}{ccccccc}
 1 & 0 & 0 & 0 & 0 & 0 & 0 \\
 0 & 1 & 1 & 1 & 1 & 1 & 1 \\
 0 & 1 & 4 & 7 & 10 & 13 & 16 \\
 0 & 1 & 7 & 20 & 40 & 67 & 101 \\
 0 & 1 & 10 & 40 & 106 & 223 & 406 \\
 0 & 1 & 13 & 67 & 223 & 572 & 1236 \\
 0 & 1 & 16 & 101 & 406 & 1236 & 3114 \\
\end{array}
\right).$$
The principal minor sequence of this matrix begins
$$1, 1, 3, 21, 315, 9765, 615195,\ldots$$ which we conjecture coincides with the $2$-factorial numbers \seqnum{A005329}. These are given by $\prod_{i=1}^n 2^{i-1}$. They enumerate upper triangular $n \times n$ $(0,1)$-matrices with no zero rows. Multiplying on the left and right respectively by the binomial matrix $\left(\binom{n}{k}\right)$ \seqnum{A007318} and its transpose, respectively, we obtain the matrix that begins
$$\left(
\begin{array}{ccccccc}
 1 & 1 & 1 & 1 & 1 & 1 & 1 \\
 1 & 2 & 4 & 8 & 16 & 32 & 64 \\
 1 & 4 & 13 & 37 & 97 & 241 & 577 \\
 1 & 8 & 37 & 132 & 410 & 1170 & 3154 \\
 1 & 16 & 97 & 410 & 1451 & 4619 & 13699 \\
 1 & 32 & 241 & 1170 & 4619 & 16138 & 51960 \\
 1 & 64 & 577 & 3154 & 13699 & 51960 & 179969 \\
\end{array}
\right).$$ By construction, this will have the same principal minor sequence.
The original generating function  expands to give the matrix that begins
$$\left(
\begin{array}{ccccccc}
 1 & 2 & 3 & 4 & 5 & 6 & 7 \\
 2 & 7 & 15 & 26 & 40 & 57 & 77 \\
 3 & 15 & 43 & 94 & 175 & 293 & 455 \\
 4 & 26 & 94 & 251 & 555 & 1079 & 1911 \\
 5 & 40 & 175 & 555 & 1431 & 3191 & 6391 \\
 6 & 57 & 293 & 1079 & 3191 & 8065 & 18109 \\
 7 & 77 & 455 & 1911 & 6391 & 18109 & 45207 \\
\end{array}
\right),$$ whose principal minor sequence begins $$1, 3, 21, 315, 9765, 615195,\ldots.$$
The diagonal sums of this matrix will have generating function
$\frac{1}{(1-2x)(1-2x-x^2)}$ which gives the convolution of $2^n$ and the Pell numbers. This is \seqnum{A094706}.
\end{example}
\begin{example} We can recover the Robbins numbers $A_n$ in determinantal form as follows. Instead of beginning with the matrix $\left(\binom{n+k}{k}\right)=\left(\binom{n}{k}\right)\left(\binom{n}{k}\right)^T$, we form the matrix $$\left(\binom{n-1}{n-k}\right)\left(\binom{n-1}{n-k}\right)^T$$ to obtain the matrix with generating function $1+\frac{xy}{1-x-y}$. We then form the matrix that begins
$$\left(
\begin{array}{ccccccc}
 1 & 0 & 0 & 0 & 0 & 0 & 0 \\
 0 & 1 & 1 & 1 & 1 & 1 & 1 \\
 0 & 1 & 2 & 3 & 4 & 5 & 6 \\
 0 & 1 & 3 & 6 & 10 & 15 & 21 \\
 0 & 1 & 4 & 10 & 20 & 35 & 56 \\
 0 & 1 & 5 & 15 & 35 & 70 & 126 \\
 0 & 1 & 6 & 21 & 56 & 126 & 252 \\
\end{array}
\right)-\left(
\begin{array}{ccccccc}
 0 & 0 & 0 & 0 & 0 & 0 & 0 \\
 0 & 0 & 1 & 0 & 0 & 0 & 0 \\
 0 & 0 & 0 & 1 & 0 & 0 & 0 \\
 0 & 0 & 0 & 0 & 1 & 0 & 0 \\
 0 & 0 & 0 & 0 & 0 & 1 & 0 \\
 0 & 0 & 0 & 0 & 0 & 0 & 1 \\
 0 & 0 & 0 & 0 & 0 & 0 & 0 \\
\end{array}
\right),$$  or
$$\left(
\begin{array}{ccccccc}
 1 & 0 & 0 & 0 & 0 & 0 & 0 \\
 0 & 1 & 0 & 1 & 1 & 1 & 1 \\
 0 & 1 & 2 & 2 & 4 & 5 & 6 \\
 0 & 1 & 3 & 6 & 9 & 15 & 21 \\
 0 & 1 & 4 & 10 & 20 & 34 & 56 \\
 0 & 1 & 5 & 15 & 35 & 70 & 125 \\
 0 & 1 & 6 & 21 & 56 & 126 & 252 \\
\end{array}
\right).$$
This matrix has generating function
$$1+\frac{xy}{1-x-y}-\left(\frac{y}{1-xy}-y\right)=\frac{1-y-x(1-y)(1+y+y^2)+x^2y}{(1-xy)(1-x-y)}.$$
The principal minor sequence of this matrix then yields the Robbins numbers $A_n$.
\end{example}

\begin{example}\label{Ex} It is possible to use a Riordan array and the symmetrization process to arrive at a matrix whose principal minor sequence begins $1,1,2,7,42,\ldots$. For this, we start with the Riordan array
$\left(\frac{1-x+x^2}{1-x}, \frac{x}{1-x}\right)$ (essentially \seqnum{A072405}) which begins
$$\left(
\begin{array}{ccccccc}
 1 & 0 & 0 & 0 & 0 & 0 & 0 \\
 0 & 1 & 0 & 0 & 0 & 0 & 0 \\
 1 & 1 & 1 & 0 & 0 & 0 & 0 \\
 1 & 2 & 2 & 1 & 0 & 0 & 0 \\
 1 & 3 & 4 & 3 & 1 & 0 & 0 \\
 1 & 4 & 7 & 7 & 4 & 1 & 0 \\
 1 & 5 & 11 & 14 & 11 & 5 & 1 \\
\end{array}
\right).$$
The inverse of this is the Riordan array $\left(\frac{1+x}{1+x+x^2}, \frac{x}{1+x}\right)$ \seqnum{A106509}, which begins
$$\left(
\begin{array}{ccccccc}
 1 & 0 & 0 & 0 & 0 & 0 & 0 \\
 0 & 1 & 0 & 0 & 0 & 0 & 0 \\
 -1 & -1 & 1 & 0 & 0 & 0 & 0 \\
 1 & 0 & -2 & 1 & 0 & 0 & 0 \\
 0 & 1 & 2 & -3 & 1 & 0 & 0 \\
 -1 & -1 & -1 & 5 & -4 & 1 & 0 \\
 1 & 0 & 0 & -6 & 9 & -5 & 1 \\
\end{array}
\right).$$ We now apply the symmetrization process to this array to get a matrix that begins
$$\left(
\begin{array}{ccccccc}
 1 & 1 & 1 & 1 & 1 & 1 & 1 \\
 1 & 0 & -1 & -2 & -3 & -4 & -5 \\
 1 & -1 & -1 & 0 & 2 & 5 & 9 \\
 1 & -2 & 0 & 1 & 1 & -1 & -6 \\
 1 & -3 & 2 & 1 & 0 & -1 & 0 \\
 1 & -4 & 5 & -1 & -1 & -1 & 0 \\
 1 & -5 & 9 & -6 & 0 & 0 & 1 \\
\end{array}
\right).$$ The generating function for this matrix is then
$$\frac{1+xy}{(1-x+xy)(1-y+xy)}.$$
This matrix has a principal minor sequence that begins
$$1, -1, -2, 7, 42, -429, -7436, 218348, 10850216, -911835460, \ldots.$$
Multiplying the columns of this matrix by $(-1)^n$ then produces a matrix whose principal minor sequence begins
$$1, 1, 2, 7, 42, 429, 7436, 218348, 10850216, 911835460, \ldots.$$
Multiplying this latter matrix on the right by the transpose of the Riordan array $\left(\frac{1+x}{1-x}, x\right)$ we obtain the matrix that begins
$$\left(
\begin{array}{ccccccc}
 1 & 1 & 1 & 1 & 1 & 1 & 1 \\
 1 & 2 & 1 & 2 & 1 & 2 & 1 \\
 1 & 3 & 3 & 2 & 4 & 1 & 5 \\
 1 & 4 & 6 & 5 & 5 & 7 & 2 \\
 1 & 5 & 10 & 11 & 10 & 11 & 12 \\
 1 & 6 & 15 & 21 & 21 & 21 & 22 \\
 1 & 7 & 21 & 36 & 42 & 42 & 43 \\
\end{array}
\right).$$
This matrix has its generating function given by
$$\frac{(1+y)(1-xy)}{(1-y)(1+y-xy)(1-x-xy)}.$$
This again has a principal minor sequence that begins
$$1, 1, 2, 7, 42, 429, 7436, 218348, 10850216, 911835460, \ldots.$$
 We note that this latter matrix is made up of a reversed copy of the Riordan array $$\left(\frac{1-x}{1-3x+3x^2-2x^3}, \frac{x}{1-x}\right),$$ which begins
$$\left(
\begin{array}{ccccccc}
 1 & 0 & 0 & 0 & 0 & 0 & 0 \\
 2 & 1 & 0 & 0 & 0 & 0 & 0 \\
 3 & 3 & 1 & 0 & 0 & 0 & 0 \\
 5 & 6 & 4 & 1 & 0 & 0 & 0 \\
 10 & 11 & 10 & 5 & 1 & 0 & 0 \\
 21 & 21 & 21 & 15 & 6 & 1 & 0 \\
 43 & 42 & 42 & 36 & 21 & 7 & 1 \\
\end{array}
\right),$$ and a matrix which is not a Riordan array.
\end{example}

In this section, we have seen that there is a link between the Robbins numbers $A_{n+1}$ and certain Riordan arrays. In the sequel, we will find links to generating functions expressed as continued fractions, and to Hankel matrices. Thus we devote the next section to a brief overview of these topics.

We finish this section with an example which uses some of these ideas.
\begin{example} For this example we let $g(x)$ be the generating function $g(x)=\frac{1-x}{1-3x^2+x^3}$. We consider the symmetric matrix with generating function $$\frac{x-y}{xg(x)-yg(y)}.$$
This expands to give a matrix that begins
$$\left(
\begin{array}{ccccccc}
 1 & -1 & 1 & -1 & 1 & -1 & 1 \\
 -1 & -1 & 2 & -3 & 4 & -5 & 6 \\
 1 & 2 & 0 & -2 & 5 & -9 & 14 \\
 -1 & -3 & -2 & 1 & 1 & -6 & 15 \\
 1 & 4 & 5 & 1 & -1 & 0 & 6 \\
 -1 & -5 & -9 & -6 & 0 & 0 & 0 \\
 1 & 6 & 14 & 15 & 6 & 0 & 1 \\
\end{array}
\right).$$
The principal minors of this matrix begin
$$1, -2, -7, 42, 429, -7436, -218348, 10850216, \ldots.$$
Now we multiply on the left and right respectively by the Riordan array $\left(\frac{g(x)}{1-x},x\right)$ and its transpose to obtain the matrix of the last example that begins
$$\left(
\begin{array}{ccccccc}
 1 & 1 & 1 & 1 & 1 & 1 & 1 \\
 1 & -1 & -2 & -3 & -4 & -5 & -6 \\
 1 & -2 & 0 & 2 & 5 & 9 & 14 \\
 1 & -3 & 2 & 1 & -1 & -6 & -15 \\
 1 & -4 & 5 & -1 & -1 & 0 & 6 \\
 1 & -5 & 9 & -6 & 0 & 0 & 0 \\
 1 & -6 & 14 & -15 & 6 & 0 & 1 \\
\end{array}
\right).$$
This is because
$$\frac{g(x)}{1-x}  \frac{x-y}{xg(x)-yg(y)}\frac{g(y)}{1-y}=\frac{1}{(1-x+xy)(1-y+xy)}.$$
The sequence with generating function $g(x)$ begins
$$1, -1, 3, -4, 10, -15, 34,\ldots.$$  This is a signed version of \seqnum{A188022}. The reversion of the generating function $f(x)=xg(x)$, that is, the solution $u(x)$ of the equation $f(u)=x$ that satisfies $u(0)=0$,  expands to give the sequence that begins
$$0, 1, 1, -1, -6, -8, 15, 84,\ldots.$$
We now take the Hankel transform $h_n$ of the sequence $t_n$ that begins $1, 1, -1, -6, -8, 15, 84,\ldots$. By definition, this is given by
$$h_n=|t_{i+j}|_{0 \le i,j \le n}.$$
We find that the Hankel transform sequence $h_n$ begins
$$1, -2, -7, 42, 429, -7436, -218348, 10850216, \ldots.$$

Sequences in this note are referred to by their A$nnnnnn$ number in the On-Line Encyclopedia of Integer Sequences, if recorded there \cite{SL1, SL2}.
\end{example}
\begin{figure}
\begin{center}
\begin{tikzpicture}
\planepartition{{5,3,2,2},{4,2,2,1},{2,1},{1}}
\end{tikzpicture}
\end{center}
\caption{A plane partition}
\end{figure}
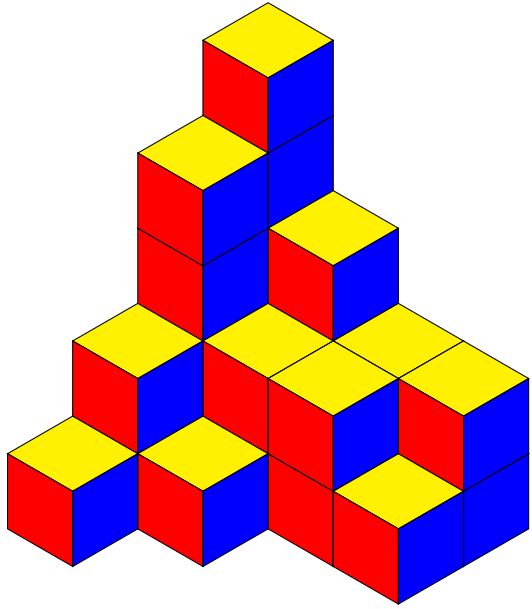

\section{Transforms, Riordan arrays, continued fractions, and Hankel transforms}
In this note we shall use a number of sequence transforms, which we now describe.
\begin{enumerate}
\item \textbf{The binomial transform}. Given a sequence $a_n$, its binomial transform is given by $b_n=\sum_{k=0}^n \binom{n}{k}a_k$. If the sequence $a_n$ has a generating function $g(x)$, then the generating function of $b_n$ is given by $\frac{1}{1-x} g\left(\frac{x}{1-x}\right)$. More generally, the $r$-th binomial transform is given by $\sum_{k=0}^n \binom{n}{k}r^{n-k}a_k$ with generating function $\frac{1}{1-rx} g\left(\frac{x}{1-rx}\right)$. The inverse binomial transform of $a_n$ is given by $\sum_{k=0}^n \binom{n}{k}(-1)^{n-k}a_k$.
\item \textbf{The INVERT$(r)$ transform}. Given a sequence $a_n$ with generating function $g(x)$, the INVERT$(r)$ transform of $a_n$ is the sequence with generating function $\frac{g(x)}{1-rxg(x)}$. We call the INVERT$(1)$ transform the invert transform.
\item \textbf{The reversion transform} Given a sequence $a_n$ with $a_0=0, a_1 \ne 0$, and generating function $f(x)$, the reversion transform of this sequence is the sequence given by the reversion of the power series $f(x)$. That is, the reversion transform of $a_n$ is the expansion of the solution $u(x)$ to the equation $f(u)=x$ which satisfies $u(0)=0$. We write $u(x)=\bar{f}(x)=\Rev(f(x))$. This transform is involutory.
\item \textbf{The revert transform} The \emph{revert} transform $b_n$ of the sequence $a_n$, where $a_0 \ne 0$, with generating function $g(x)$, is given by expansion of $\frac{1}{x} \Rev(xg(x))$. We use the notations $\rev(a_n)$ and $\rev(g(x))$ for this. We have $$b_n = \frac{1}{n+1} [x^n]\frac{1}{g(x)^{n+1}}.$$ This last equation is a consequence of Lagrange Inversion \cite{H2, LI}. This transform is involutory.
\item \textbf{The log revert transform} The \emph{log revert} transform $b_n$ of the sequence $a_n, a_0 \ne 0$ with generating function $g(x)$ is given by the expansion of $\frac{x \frac{d}{dx}\Rev(xg(x))}{\Rev(xg(x))}$. We have $$b_n=[x^n] \frac{1}{g(x)^n}.$$
\end{enumerate}
Note that the term ``log revert transform'' is not a standard term, but we use it here for convenience. It could also justifiably be termed a ``hitting time'' transform.
\begin{example}
We have
$$\frac{1}{n+1} [x^n] ((1+x)^3)^{n+1}=\frac{1}{2n+3} \binom{3n+3}{n+1} \Longrightarrow 1,3,26,646,\ldots,$$
and
$$ [x^n] ((1+x)^3)^n = \binom{3n}{n} \Longrightarrow 3^n[1,2,11,270,7429,\ldots].$$
Here, we have used the symbol $\Longrightarrow$ to indicate that the sequence has the following Hankel transform. The notation $r^n[a_0, a_1, a_2,\ldots]$ indicates the sequence $r^0a_0, r^1a_1, r^2a_2,\ldots$.
\end{example}
To describe Riordan arrays \cite{book, SGWW}, we consider a field $\mathbb{F}$ of characteristic $0$, and we define two subsets of the ring $\mathbb{F}[[x]]$ of formal power series with coefficients in $\mathbb{F}$ defined by
$$\mathcal{F}_0=\left\{ g(x) \in \mathbb{F}[[x]]\,|\, g(x)= g_0+ g_1 x+ g_2x^2+g_3 x^3+\cdots, g_0 \ne 0\right\},$$
and
$$\mathcal{F}_1=\left\{ f(x) \in \mathbb{F}[[x]]\,|\, f(x)= f_1 x+ f_2x^2+f_3 x^3+\cdots, f_1 \ne 0\right\}.$$

The elements $g(x)$ in $\mathcal{F}_0$ are thus invertible for multiplication - that is, $\frac{1}{g(x)}$ exists and is again in $\mathcal{F}_0$. The elements $f(x)$ in $\mathcal{F}_1$ are composable with a compositional inverse (or reversion), given by $\bar{f}(x)$ where $u(x)=\bar{f}(x)=\Rev(f)(x)$ is the solution $u(x)$ of the equation
$f(u)=x$ that satisfies $u(0)=0$.
\begin{example} For $f(x)=x(1-x)=x-x^2$ we have $\bar{f}(x)=\frac{1-\sqrt{1-4x}}{2}$. This follows since the solution to $$u-u^2=x$$ or $$u^2-u+x=0$$ is given by $u=\frac{1 \pm \sqrt{1-4x}}{2}$, and we must choose the minus sign in order to have $u(0)=0$.
\end{example}
By a Riordan array we mean a pair $(g(x), f(x)) \in \mathcal{F}_0 \times \mathcal{F}_1$. Associated with this couple is a matrix with elements $t_{n,k}$ in $\mathbb{F}$ given by
$$t_{n,k}=[x^n] g(x)f(x)^k.$$ Here, $[x^n]$ is the functional on $\mathbb{F}[[x]]$ that extracts the coefficient of $x^n$ in the expansion of a power series. We often switch between these two views of a Riordan array when no confusion arises. The bivariate generating function of the matrix that represents $(g(x), f(x))$ is given by
$$B(x,y)=\frac{g(x)}{1-y f(x)}.$$ The triangle whose bivariate generating function is given by
$$ \rev_x\{\frac{g(x)}{1-y f(x)}\}$$ is called the \emph{inversion} of the original Riordan array.
The set of Riordan arrays is a group, where the group multiplication is given by
$$(g(x), f(x))\cdot (u(x), v(x))= (g(x)u(f(x)), v(f(x)),$$ and the unit is given by $(1, x)$. The inverse of the array $(g(x), f(x))$ is given by
$$(g(x), f(x))^{-1} = \left(\frac{1}{g(\bar{f})}, \bar{f}(x)\right).$$
The symmetrization of a Riordan array $(g(x), f(x))$ is the matrix with generating function
$$B\left(xy, \frac{1}{x}\right)+B\left(xy, \frac{1}{y}\right)-g(xy).$$
With this structure, the group of Riordan arrays is none other than the semi-direct product of $\mathcal{F}_0$ with $\mathcal{F}_1$.

Riordan arrays have a rich structure, evidenced for instance by the fact that embedded in a given Riordan array are many other Riordan arrays. For example, the array whose elements are given by $t_{2n-k, n}$ is again a Riordan array, called the \emph{vertical half} of the array $(g(x), f(x))$. It is in fact the Riordan array given by
$$V=\left(\frac{x \phi'(x) g(\phi(x))}{\phi(x)}, \phi(x)\right),$$ where
$$\phi(x)=\text{Rev}\left(\frac{x^2}{f(x)}\right).$$
\begin{example} Pascal's triangle $\left(\binom{n}{k}\right)$ is a Riordan array, given by $\left(\frac{1}{1-x}, \frac{x}{1-x}\right)$. Its vertical half $\left(\binom{2n-k}{n}\right)$ is the Riordan array
$$\left(\frac{1}{\sqrt{1-4x}}, xc(x)\right).$$  Multiplying this by the Riordan array $\left(\frac{1}{1-x}, x\right)$ yields the Riordan array $\left(\frac{1}{(1-x)\sqrt{1-4x}}, xc(x)\right)$ whose principal minors give us the Robbins numbers.

We note that the symmetrization of $\left(\binom{n}{k}\right)$ is given by $\left(\binom{n+k}{k}\right)$.
\end{example}
The Riordan arrays we have considered heretofore have been defined by generating functions that are so-called \emph{ordinary} generating functions. Another form of Riordan array is the $exponential$ Riordan array, again defined by two generating functions $g(x)$ and $f(x)$, but now, these are exponential generating functions of the type
$$g(x)=g_0 + g_1 \frac{x}{1!} + g_2 \frac{x^2}{2!} = g_3 \frac{x^3}{3!} + \ldots, \quad g_0 \ne 0,$$ and
$$f(x)=f_1 \frac{x}{1!}+ f_2\frac{x^2}{2!} + \cdot, \quad f_0=0, f_1 \ne 0.$$
In this case, we have a matrix representation of the pair $(g(x), f(x))$ where the $(n,k)$-th element of the corresponding matrix is given by
$$t_{n,k}= \frac{n!}{k!} g(x)f(x)^k.$$ To distinguish from ordinary Riordan arrays, we use the notation $[g(x), f(x)]$ to denote this exponential Riordan array.

Elements of the (ordinary) Riordan group of the form $(g(x), xg(x))$ are called Bell matrices. We have the following result \cite{Inv}.
\begin{proposition} The inversion of the (ordinary) Bell matrix $(g(x), xg(x))$ is given by the exponential Riordan array $\left[\left(\rev(g(x))\right)_e, -x\right]$, where $G_e(x)$ is the exponential generating function of the expansion of $G(x)$.
\end{proposition}
The \emph{Hankel transform} \cite{Kratt1, Kratt2, Layman} of a sequence $a_n$ is the sequence $h_n$ where $$h_n=|a_{i+j}|_{0 \le i,j \le n}.$$  Note that we use a $n+1 \times n+1$ matrix for this (other authors use a $n \times n$ matrix). If the g.f. of $a_n$ is expressible as a continued fraction of the form
$$\cfrac{\mu_0}{1-\alpha_0 x-
\cfrac{\beta_1 x^2}{1-\alpha_1 x -
\cfrac{\beta_2 x^2}{1-\alpha_2 x -
\cfrac{\beta_3 x^2}{1-\alpha_3 x -\cdots}}}},$$
then the Hankel transform of $a_n$ is given by \cite{Kratt1} the Heilermann formula
$$h_n=\mu_0^{n+1}\beta_1^n\beta_2^{n-1}\cdots\beta_{n-1}^2\beta_n.$$

If the g.f. of $a_n$ is expressible as the following type of continued fraction:
$$\cfrac{\mu_0}{1+
\cfrac{\gamma_1 x}{1+
\cfrac{\gamma_2 x}{1+\cdots}}},$$ then we have
$$h_n=\mu_0^{n+1}(\gamma_1 \gamma_2)^n (\gamma_3 \gamma_4)^{n-1} \cdots (\gamma_{2n-3} \gamma_{2n-2})^2 (\gamma_{2n-1}\gamma_{2n}).$$

Given a sequence $a_n$ with generating function $g(x)$, then the following sequences will have the same Hankel transform:
\begin{enumerate}
\item $(-1)^n a_n$, with generating function $g(-x)$,
\item The $r$-th binomial transform $\sum_{=0}^n \binom{n}{k}r^{n-k}a_k$ of $a_n$, with generating function $\frac{1}{1-rx}g\left(\frac{x}{1-rx}\right)$,
\item The $r$-invert transform of $a_n$, with generating function $\frac{g(x)}{1-r g(x)}$.
\end{enumerate}
Thus any combination of these, applied to a sequence, will leave the Hankel transform unchanged.

The \emph{revert transform} of $g(x)$, where $g(0)\ne 0$, is given by $\frac{1}{x}\text{Rev}(xg(x))$. We then have the following.
\begin{enumerate}
\item The revert transform of the binomial transform of $[x^n]g(x)$ is the invert transform of the revert transform of $[x^n]g(x)$ \cite{Inv},
\item The revert transform of the invert transform of $[x^n]g(x)$ is the binomial transform of the revert transform of $[x^n]g(x)$.
\end{enumerate}
A consequence of this is the following result.
\begin{proposition} The revert transforms of the sequences with generating function given by a continued fraction of the form
$$\cfrac{1}{1- (\alpha+s) x - \cfrac{\beta x^2}{1-(\alpha+r) x - \cfrac{\beta x^2}{1-(\alpha+r) x}}}$$ all have the same Hankel transform, for different values of $r$ and $s$.
\end{proposition}
\begin{example}
We have
$$\frac{1-3x+2x^2}{1-4x+3x^2+x^3}=\cfrac{1}{1-x-\cfrac{x^2}{1-\frac{3}{2}x-\cfrac{\frac{1}{4}x^2}{1-\frac{3}{2}x}}}.$$ This generating function expands to give the sequence \seqnum{A121449} which begins
$$1, 1, 3, 8, 22, 61, 170, 475, 1329, 3721, 10422, 29196,\ldots.$$
We have
$$\frac{1-x}{1-3x+x^3}=\cfrac{1}{1-2x-\cfrac{x^2}{1-\frac{1}{2}x-\cfrac{\frac{1}{4}x^2}{1-\frac{1}{2}x}}}.$$
This expands to give the sequence \seqnum{A052536} which counts compositions of $n$ when parts $1$ and $2$ are of two kinds.
The Hankel transform of the revert transforms of these sequences are both equal to the sequence that begins
$$1, -2, -7, 42, 429, -7436,\ldots.$$
Note that the Hankel transform of the revert transform of $\frac{1-ix}{1-3ix-x^3}$ will then have a Hankel transform that begins $$1, 2, 7, 42, 429, 7436,\ldots.$$
The sequence with generating function $$\frac{1}{1-4x+3x^2+x^3}=\cfrac{1}{1-4x+\cfrac{3x^2}{1-\frac{1}{3}x+\cfrac{\frac{1}{9}x^2}{1+\frac{1}{3}x}}},$$ begins
$$1, 4, 13, 39, 113, 322, 910, 2561, 7192, 20175, \ldots.$$ This is \seqnum{A215404}.
Its revert transform has its Hankel transform equal to the sequence
$$ 1, 3, 26, 646, 45885, 9304650, 5382618660, 8878734657276, \ldots.$$ This is
\seqnum{A005156}$(n+1)$, where \seqnum{A005156}  counts the number of alternating sign $(2n+1) \times (2n+1)$ matrices symmetric about the vertical axis. This can be seen since the generating function $(1-x)^2$ of the revert transform of $\frac{1}{n+1}\binom{3n+1}{n+1}$ \seqnum{A006013} has the continued fraction expression
$$(1-x)^2=\cfrac{1}{1+2x+\cfrac{3x^2}{1-\frac{4}{3}x+\cfrac{\frac{1}{9}x^2}{1-\frac{2}{3}x}}}.$$ Alternatively, we can note that $\frac{1}{(1-x)^3}$, the generating function of the revert transform of $\frac{1}{2n+3}\binom{3n+3}{n+1}(-1)^n$, has its generating function given by
$$\frac{1}{(1-x)^3}=\cfrac{1}{1-3x+\cfrac{3x^2}{1+\frac{1}{3}x+\cfrac{\frac{1}{9}x^2}{1-\frac{1}{3}x}}}.$$
\end{example}
\begin{example} \emph{The Lawrence challenge.} In 2011, Peter Lawrence issued the following challenge: ``find a $3 \times 3$ integer matrix with ``smallish'' elements whose powers generate a sequence that is not in the OEIS''. The recorded answer to this challenge is \seqnum{A200715}. This is essentially the sequence that begins
$$1, 1, -2, -4, 3, 13, 0, -36, -23, 85, 118, -160, -429, 16.\ldots$$ It is generated by following the $(2,3)$ element of the sequences of matrices
$$\left(
\begin{array}{ccc}
 0 & 1 & 0 \\
 0 & 0 & 1 \\
 1 & -3 & 1 \\
\end{array}
\right)^n.$$
This sequence has generating function
$$\frac{1}{1-x+3x^2-x^3}=\cfrac{1}{1+x+\cfrac{3x^2}{1-\frac{1}{3}x+\cfrac{\frac{1}{9}x^2}{1+\frac{1}{3}x}}}.$$
Its revert transform begins
$$1, -1, 4, -11, 41, -146, 564, -2199, 8835, -35989, 148912,\ldots$$ which is an alternating sign version of
\seqnum{A030981}, which counts the number of rooted non-crossing trees with $n$ nodes and no non-root nodes of degree $1$ (E. Deutsch). The Hankel transform of this latter sequence is again
$$1, 3, 26, 646, 45885, 9304650,\ldots.$$
\end{example}

In general, the generating function of the Hankel matrix associated to the expansion of $g(x)$ is given by
$$\frac{x g(x)- y g(y)}{x-y}.$$ When $g(x)=\frac{1}{x}\Rev(f(x))$, then this takes the form
$$\frac{\Rev(f(x))-\Rev(f(y))}{x-y}.$$ We apply the fundamental theorem of Riordan arrays as follows.
\begin{align*} (1, f(y))\cdot(1, f(x))\frac{\Rev(f(x))-\Rev(f(y))}{x-y}
&=(1, f(y)) \frac{\Rev(f(x))(f(x))- \Rev(f(y))}{f(x)-y}\\
&=\frac{\Rev(f(x))(f(x))- \Rev(f(y))(f(y))}{f(x)-f(y)}\\
&=\frac{x-y}{f(x)-f(y)}.\end{align*}
We thus have the following result.
\begin{proposition} Let $g(x)$ be the revert transform of $G(x)$. Then the Hankel transform of the expansion of $g(x)$ is equal to the principal minor sequence of the symmetric matrix with generating function
$$\frac{x-y}{xG(x)-yG(y)}.$$\end{proposition}
\begin{proof} That the determinant sequences are equal follows from the fact that the matrix representations of the Riordan arrays $(1, f(x))$ and $(1, f(y))$ above are lower triangular matrices with $1$'s on the diagonal. Here, $f(x)=xG(x)$.
\end{proof}
In many cases, the generating function $G(x)$ may be simpler to deal with than $g(x)$.  This technique was used by Gessel and Xin \cite{Gessel} in studying the Hankel transform of the ternary numbers.  \section{Centered polygon numbers}
The centered polygon numbers are the number sequences with generating function
$$\frac{1+(r-2) x+x^2}{(1-x)^3}.$$ For $r=3,4,5$ they are called, respectively, the triangular, square and pentagonal polygon numbers.
\begin{center}
\begin{tabular}{|c|c|c|c|c|c|c|c|}\hline
$r=0$ & $1$ & $1$ & $1$ & $1$ & $1$ & $1$ & \seqnum{A000012}\\ \hline
$r=1$ & $1$ & $2$ & $4$ & $7$ & $11$ & $16$ & \seqnum{A000124}\\ \hline
$r=2$ & $1$ & $3$ & $7$ & $13$ & $21$ & $31$ & \seqnum{A002061}\\ \hline
$r=3$ & $1$ & $4$ & $10$ & $19$ & $31$ & $46$ & \seqnum{A005448}\\ \hline
$r=4$ & $1$ & $5$ & $13$ & $25$ & $41$ & $61$ & \seqnum{A001844}\\ \hline
$r=5$ & $1$ & $6$ & $16$ & $31$ & $51$ & $76$ & \seqnum{A005891}\\ \hline
\end{tabular}
\end{center}

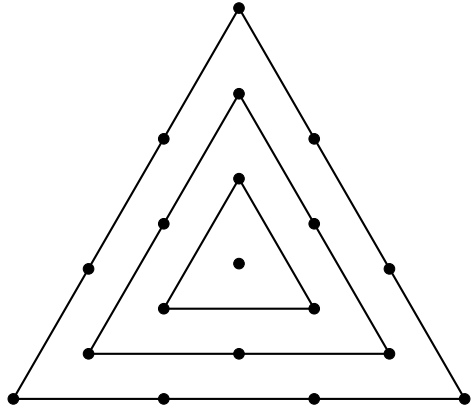
\begin{figure}
\begin{center}
\begin{tikzpicture}[scale=1]
\draw[black, thick] (0,0)--(6,0)--(3,5.2) -- cycle;
\draw[black, thick] (1,0.6)--(5,0.6)--(3,4.06) -- cycle;
\draw[black, thick] (2,1.2)--(4,1.2) --(3,2.93) -- cycle;

\filldraw[black] (0,0) circle (2pt);
\filldraw[black] (6,0) circle (2pt);
\filldraw[black] (3,5.2) circle (2pt);
\filldraw[black] (1,0.6) circle (2pt);
\filldraw[black] (5,0.6) circle (2pt);
\filldraw[black] (3,4.06) circle (2pt);
\filldraw[black] (2,1.2) circle (2pt);
\filldraw[black] (4,1.2) circle (2pt);
\filldraw[black] (3,2.93)  circle (2pt);

\filldraw[black] (3,1.8)  circle (2pt);
\filldraw[black] (3,0.6)  circle (2pt);
\filldraw[black] (2,2.33) circle(2pt);
\filldraw[black] (4,2.33) circle(2pt);

\filldraw[black] (1,1.73) circle(2pt);
\filldraw[black] (5,1.73) circle(2pt);
\filldraw[black] (2,3.46) circle(2pt);
\filldraw[black] (4,3.46) circle(2pt);
\filldraw[black] (2,0) circle (2pt);
\filldraw[black] (4,0) circle (2pt);
\end{tikzpicture}
\caption{Centered triangular numbers $1,4,10,19,\ldots$}
\end{center}
\end{figure}

The sequence for general $r$ begins
$$1, r + 1, 3r + 1, 6r + 1, 10r + 1, 15r + 1, 21r + 1, 28r + 1, \ldots,$$
with general term $1+r \binom{n+1}{2}$. We note that this sequence can be obtained as
$$\sum_{k=0}^n N_{n,k}(\binom{1}{k}r-\binom{0}{k}(r-1)),$$ where $N_{n,k}=\frac{1}{k+1}\binom{n+1}{k}\binom{n}{k}$ are the Narayana numbers, or as
$$\sum_{k=0}^n \binom{n}{k}(r\binom{2}{k}-r \binom{1}{k}+\binom{0}{k}).$$

The revert transform of this sequence then begins
$$1, -r - 1, 2r^2 + r + 1, - 5r^3 - r - 1, 14r^4 - 7r^3 + r + 1, - 42r^5 + 42·r^4 - r - 1,\ldots.$$
This is a polynomial sequence with coefficient array
$$\left(
\begin{array}{cccccccc}
 1 & 0 & 0 & 0 & 0 & 0 & 0 & 0 \\
 -1 & -1 & 0 & 0 & 0 & 0 & 0 & 0 \\
 1 & 1 & 2 & 0 & 0 & 0 & 0 & 0 \\
 -1 & -1 & 0 & -5 & 0 & 0 & 0 & 0 \\
 1 & 1 & 0 & -7 & 14 & 0 & 0 & 0 \\
 -1 & -1 & 0 & 0 & 42 & -42 & 0 & 0 \\
 1 & 1 & 0 & 0 & 30 & -198 & 132 & 0 \\
 -1 & -1 & 0 & 0 & 0 & -297 & 858 & -429 \\
\end{array}
\right).$$
The Hankel transform of this revert transform will again be a polynomial sequence in $r$. In fact, this Hankel transform $h_n(r)$ begins
$$1, r(r-1), r^3(r^3 - 3r^2 + 3r - 2), r^6(r^6 - 6r^5 + 15r^4 - 24r^3 + 30r^2 - 24r + 6), \ldots.$$
Dividing by $r^{\binom{n+1}{2}}$, we find that $\frac{h_n(r)}{r^{\binom{n+1}{2}}}$ begins
$$1, r - 1, r^3 - 3r^2 + 3r - 2, r^6 - 6r^5 + 15r^4 - 24r^3 + 30r^2 - 24r + 6,\ldots.$$
We give below values for this sequence for $r=0,\ldots,4$.
\begin{center}
\begin{tabular}{|c|c|c|c|c|c|c|c|}\hline
$r=0$ & $1$ & $-1$ & $-2$ & $6$ & $33$ & $-286$ $\ldots$ \\ \hline
$r=1$ & $1$ & $0$ & $-1$ & $-2$ & $3$ & $18$ $\ldots$ \\ \hline
$r=2$ & $1$ & $1$ & $0$ & $-2$ & $-5$ & $-14$ $\ldots$ \\ \hline
$r=3$ & $1$ & $2$ & $7$ & $42$ & $429$ & $7436$ $\ldots$ \\ \hline
$r=4$ & $1$ & $3$ & $26$ & $646$ & $45885$ & $9304650$ $\ldots$ \\ \hline
\end{tabular}
\end{center}
Thus for $r=3$ (the centered triangle polygon numbers) we obtain a scaled Hankel transform that gives
$$1,2,7,42,429, 7436,\ldots.$$ The sequences for $r=0$ and $r=4$ are of importance as well in the theory of plane partitions and alternating sign matrices. We shall consider them in a separate section of this note.

We let $g(x)=\frac{1+x+x^2}{(1-x)^3}$, and we form the bivariate generating function
$$\frac{x - y}{x g(x)- y g(y)}.$$ This expands to give the symmetric matrix $M$ that begins
$$\left(
\begin{array}{ccccccc}
 1 & -4 & 6 & -3 & -3 & 6 & -3 \\
 -4 & 22 & -51 & 57 & -6 & -78 & 111 \\
 6 & -51 & 189 & -378 & 351 & 189 & -1026 \\
 -3 & 57 & -378 & 1296 & -2457 & 1809 & 3078 \\
 -3 & -6 & 351 & -2457 & 8424 & -15444 & 8073 \\
 6 & -78 & 189 & 1809 & -15444 & 54378 & -97011 \\
 -3 & 111 & -1026 & 3078 & 8073 & -97011 & 354294 \\
\end{array}
\right).$$
By construction, we obtain that the scaled principal minors of this matrix, $\frac{|M_n|}{3^{\binom{n+1}{2}}}$, yield the sequence
$$1, 2, 7, 42, 429, 7436, 218348, 10850216, 911835460,  \ldots.$$
By multiplying on the left and right, respectively, by $\left(\frac{g(x)}{1-x}, x\right)$ and its transpose, we obtain the symmetric matrix that begins
$$\left(
\begin{array}{ccccccc}
 1 & 1 & 1 & 1 & 1 & 1 & 1 \\
 1 & 7 & 4 & 1 & 7 & 4 & 1 \\
 1 & 4 & 34 & 1 & -23 & 61 & 1 \\
 1 & 1 & 1 & 163 & -80 & -242 & 568 \\
 1 & 7 & -23 & -80 & 898 & -752 & -1862 \\
 1 & 4 & 61 & -242 & -752 & 5515 & -5588 \\
 1 & 1 & 1 & 568 & -1862 & -5588 & 35884 \\
\end{array}
\right),$$ which has the same principal minor sequence.

We still have to prove that the resulting numbers $1, 2, 7, 42, 429, 7436, \ldots$ are the Robbins numbers. We do this as follows. We have
$$\frac{1+x+x^2}{(1-x)^3}=\cfrac{1}{1-4x+\cfrac{6x^2}{1-\frac{1}{2}x+\cfrac{\frac{3}{4}x^2}{1+\frac{1}{2}x}}}.$$ Thus we shall achieve the same result for any generating function of the form
$$\cfrac{1}{1-(4+s)x+\cfrac{6x^2}{1-(\frac{1}{2}+t)x+\cfrac{\frac{3}{4}x^2}{1+(\frac{1}{2}+t)x}}}.$$
The choice of $s=1$ and $t=-1$ gives us the generating function
$$\tilde{g}(x)=1+3x+3x^2.$$ In fact, this is the inverse binomial transform of $\frac{1+x+x^2}{(1-x)^3}$.
From this we can conclude that the numbers are indeed the Robbins numbers \cite{Gessel}.
\begin{proposition} The Robbins numbers $A_{n+1}$ are given by the Hankel transform of the revert transform of the centered triangle polygon numbers, scaled by $3^{\binom{n+1}{2}}$.
\end{proposition}
\section{Heptagons and nonagons}
To each regular polygon there are a number of naturally associated polynomials. For instance, if we take the side length of a regular heptagon to be $1$, and if we designate by $\rho$ and $\sigma$ the lengths of two distinct diagonals with $\sigma > \rho$, then we have
$$\rho^3-\rho^2-2 \rho+1 =0,$$ with $\rho=2 \cos\left(\frac{\pi}{7}\right)$, and $\sigma=4 \cos^2\left(\frac{\pi}{7}\right)$. Then the equation
$$x^3-x^2-2x+1=0$$ has roots $\rho, \frac{1}{\sigma}, -\frac{\sigma}{\rho}$, and the equation
$$x^3-2x^2-x+1=0$$ has roots $\sigma, \frac{1}{\rho}, -\frac{\rho}{\sigma}$.
Thus the polynomials $1-2x-x^2+x^3$ and $1-x-2x^2+x^3$ are naturally associated with the heptagon. Following \cite{Heptagon}, to each regular $n$-gon we can associate a polynomial which has $2 \cos\left(\frac{\pi}{n}\right)$ as a root. For odd $n$, this polynomial is
$$P_n(x)=\sum_{i=0}^{\lfloor \frac{n}{2} \rfloor}(-1)^i\binom{k-i}{i}x^{k-2i} - \sum_{i=0}^{\lfloor \frac{n}{2} \rfloor}(-1)^i \binom{k-i-1}{i}x^{k-(2i+1)},$$ with $P_0(x)=1$, where $k=\frac{n-1}{2}$.
These polynomials are defined by the coefficient array that is given by the Riordan array
$$\left(\frac{1-x}{1+x^2}, \frac{x}{1+x^2}\right),$$ which begins
$$\left(
\begin{array}{ccccccc}
 1 & 0 & 0 & 0 & 0 & 0 & 0 \\
 -1 & 1 & 0 & 0 & 0 & 0 & 0 \\
 -1 & -1 & 1 & 0 & 0 & 0 & 0 \\
 1 & -2 & -1 & 1 & 0 & 0 & 0 \\
 1 & 2 & -3 & -1 & 1 & 0 & 0 \\
 -1 & 3 & 3 & -4 & -1 & 1 & 0 \\
 -1 & -3 & 6 & 4 & -5 & -1 & 1 \\
\end{array}
\right).$$
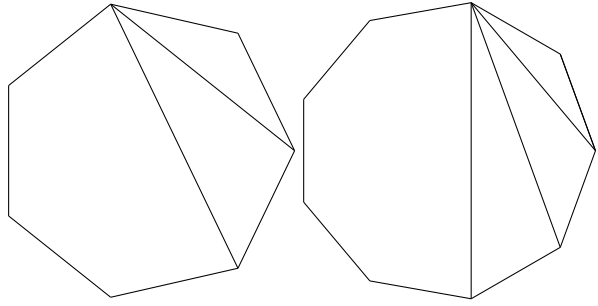
\begin{figure}
\begin{center}
\begin{tikzpicture}[scale=2]
\begin{scope}

    \draw (0 :1 cm) -- (102.85714: 1 cm);
    \draw (102.85714: 1 cm) -- (308.5714: 1 cm);
    \foreach \x in {0,51.4286,...,360} {
        \draw (\x:1 cm) -- (\x+51.4286:1 cm);
        \x = \x+51.4286;
    }
\end{scope}
\begin{scope}[xshift=2cm]
    \foreach \x in {0,40,...,360} {
        \draw (\x:1 cm) -- (\x+40:1 cm);
        \x = \x+40;
    }
    \draw (0 :1 cm) -- (80: 1 cm);
    \draw (80: 1 cm) -- (320: 1 cm);
    \draw (80: 1 cm) -- (280: 1 cm);
\end{scope}
\end{tikzpicture}
\end{center}
\caption{A heptagon and a nonagon, showing diagonals}
\end{figure}

Thus for the heptagon we have the polynomial
$$ 1-2x-x^2+x^3,$$ and for the nonagon we have the polynomial
$$1+2x-3x^2-x^3+x^4.$$
We use these polygon polynomials to form two generating functions, namely
$$\frac{1-x}{1-2x-x^2+x^3},$$ and
$$\frac{1-x^2}{1+2x-3x^2-x^3+x^4}.$$
The generating function $\frac{1-x}{1-2x-x^2+x^3}$ expands to give the sequence \seqnum{A077998} which begins
$$1, 1, 3, 6, 14, 31, 70, 157, 353, 793, 1782,\ldots.$$ This sequence counts, for instance, the number  of compositions of $n$ if there are two kinds of part $2$. This sequence corresponds to the diagonal sums of the Riordan array
$$\left(\frac{1-x}{1-2x}, \frac{x(1-x)}{1-2x}\right).$$
We have
$$\frac{1-x}{1-2x-x^2+x^3}=\cfrac{1}{1-x-\cfrac{2x^2}{1-\frac{1}{2}x-\cfrac{\frac{1}{4}x^2}{1-\frac{1}{2}x}}}.$$
The generating function $\frac{1-x^2}{1+2x-3x^2-x^3+x^4}=\frac{1+x}{1-3x+x^3}$ expands to give the sequence that begins
$$1, -2, 6, -17, 49, -141, 406, -1169, 3366, -9692,\ldots.$$ This is an alternating sign version of \seqnum{A052536}, the number of compositions of $n$ when parts $1$ and $2$ are of two kinds. We have
$$\frac{1-x^2}{1+2x-3x^2-x^3+x^4}=\cfrac{1}{1+2x-\cfrac{2x^2}{1+\frac{1}{2}x-\cfrac{\frac{1}{4}x^2}{1+\frac{1}{2}x}}}.$$
We therefore know that the revert transforms of these two sequences will have the same Hankel transform. The revert transforms begin, respectively,
$$1, -1, -1, 4, 0, -17, 16, 68, -146, -221, 1003, 273, -5939,\ldots,$$ and
$$1, 2, 2, -3, -17, -27, 30, 248, 467, -417, -4387, -9072, 6792,\ldots.$$
Both these sequences then have Hankel transform sequences that begin
$$1, -2, -7, 42, 429, -7436, -218348, 10850216, 911835460,\ldots.$$
It follows that the principal minor sequence of the matrix with generating function
$$\frac{g(x)}{1-x} \frac{x-y}{xg(x)-yg(x)} \frac{g(y)}{1-y},$$ with
$g(x)=\frac{1-x}{1-2x-x^2+x^3}$, will give the same sequence. This array begins
$$\left(
\begin{array}{ccccccc}
 1 & 1 & 1 & 1 & 1 & 1 & 1 \\
 1 & -1 & -2 & -3 & -4 & -5 & -6 \\
 1 & -2 & 0 & 2 & 5 & 9 & 14 \\
 1 & -3 & 2 & 1 & -1 & -6 & -15 \\
 1 & -4 & 5 & -1 & -1 & 0 & 6 \\
 1 & -5 & 9 & -6 & 0 & 0 & 0 \\
 1 & -6 & 14 & -15 & 6 & 0 & 1 \\
\end{array}
\right).$$
For the generating function $\frac{1-x^2}{1+2x-3x^2-x^3+x^4}$, we obtain the matrix that begins
$$\left(
\begin{array}{ccccccc}
 1 & 1 & 1 & 1 & 1 & 1 & 1 \\
 1 & -1 & 0 & -1 & 0 & -1 & 0 \\
 1 & 0 & 4 & 0 & 5 & -1 & 6 \\
 1 & -1 & 0 & -7 & 3 & -12 & 9 \\
 1 & 0 & 5 & 3 & 15 & -6 & 30 \\
 1 & -1 & -1 & -12 & -6 & -28 & 14 \\
 1 & 0 & 6 & 9 & 30 & 14 & 57 \\
\end{array}
\right).$$
\begin{example} In this example, we consider the generating function
$$g(x)=\frac{1+ix}{1+3ix-(ix)^3}.$$
The matrix with generating function
$$\frac{g(x)}{1-x} \frac{g(y)}{1-y} \frac{x-y}{xg(x)-yg(x)}$$ then begins
$$\left(
\begin{array}{ccccccc}
 1 & 1 & 1 & 1 & 1 & 1 & 1 \\
 1 & 3 & 3-i & 2-i & 2 & 3 & 3-i \\
 1 & 3-i & 6-2 i & 5-5 i & 1-4 i & i+2 & 8 \\
 1 & 2-i & 5-5 i & 7-8 i & 3-12 i & -7 i-5 & 6 i+1 \\
 1 & 2 & 1-4 i & 3-12 i & 1-16 i & -17 i-7 & -4 i-16 \\
 1 & 3 & i+2 & -7 i-5 & -17 i-7 & -18 i-14 & -11 i-23 \\
 1 & 3-i & 8 & 6 i+1 & -4 i-16 & -11 i-23 & -4 i-31 \\
\end{array}
\right).$$
The principal minor sequence of this matrix begins
$$1, 2, 7, 42, 429, 7436, 218348, \ldots.$$
\end{example}
Some sequences in this section, and closely related sequences, have been studied previously under the name of so-called ``quasi-Fibonacci'' sequences \cite{Wang, Witula}.

\section{A special matrix}
It is well known that the powers of the matrix $\left(
\begin{array}{cc}
 1 & 1 \\
 1 & 0 \\
\end{array}
\right)$ generate the Fibonacci numbers \seqnum{A000045}. This matrix is also closely connected to the theory of partial fractions \cite{CF}. The matrices
$$\left(
\begin{array}{ccc}
 1 & 1 & 1 \\
 1 & 1 & 0 \\
 1 & 0 & 0 \\
\end{array}
\right)^n $$ generate the sequence \seqnum{A077998}
$$1, 1, 3, 6, 14, 31, 70, 157, 353, 793, 1782,\ldots,$$ in the $(1,1)$ position. This sequence has generating function
$$\frac{1-x}{1-2x-x^2+x^3}=\cfrac{1}{1-x-\cfrac{2x^2}{1-\frac{1}{2}x-\cfrac{\frac{1}{4}x^2}{1-\frac{1}{2}x}}}.$$ The revert transform of this sequence, which begins
$$1, -1, -1, 4, 0, -17, 16, 68, -146, -221, 1003,\ldots,$$ has a Hankel transform that begins
$$1, -2, -7, 42, 429, -7436,\ldots.$$
We now parameterize this situation, by looking at the polynomial sequence that is generated in the $(1,1)$-position in the powers of the matrix
$$\left(
\begin{array}{ccc}
 r & 1 & 1 \\
 1 & 1 & 0 \\
 1 & 0 & 0 \\
\end{array}
\right).$$
This polynomial sequence $P_n(r)$ begins
$$1, r, r^2 + 2, r^3 + 4r + 1, r^4 + 6r^2 + 2r + 5, r^5 + 8r^3 + 3r^2 + 14r + 5, \ldots,$$ with coefficient array that begins
$$\left(
\begin{array}{ccccccc}
 1 & 0 & 0 & 0 & 0 & 0 & 0 \\
 0 & 1 & 0 & 0 & 0 & 0 & 0 \\
 2 & 0 & 1 & 0 & 0 & 0 & 0 \\
 1 & 4 & 0 & 1 & 0 & 0 & 0 \\
 5 & 2 & 6 & 0 & 1 & 0 & 0 \\
 5 & 14 & 3 & 8 & 0 & 1 & 0 \\
 14 & 14 & 27 & 4 & 10 & 0 & 1 \\
\end{array}
\right).$$ This is the Riordan array
$$\left(\frac{1-x}{1-x-2x^2+x^3}, \frac{x(1-x)}{1-x-2x^2+x^3}\right).$$
The polynomial sequence has generating function
$$\frac{1-x}{1-(r+1)x+(r-2)x^2+x^3}=\cfrac{1}{1-rx-\cfrac{2x^2}{1-\frac{1}{2}x-\cfrac{\frac{1}{4}x^2}{1-\frac{1}{2}x}}}.$$
This guarantees that the revert transform of the polynomial sequence will have Hankel transform
$$1, -2, -7, 42, 429, -7436,\ldots,$$ for any value of $r$.
The revert transform of the polynomial sequence $\bar{P}_n(r)$ begins
$$1, -r, r^2 - 2, - r^3 + 6r - 1, r^4 - 12r^2 + 4r + 7, - r^5 + 20r^3 - 10r^2 - 35r + 9,\ldots.$$
We then have
$$\bar{P}_n(r)=\frac{1}{n+1}[x^n] \frac{1}{\left(\frac{1-x}{1-(r+1)x-(r-2)x^2+x^3}\right)^{n+1}}.$$
This polynomial sequence has a coefficient array that begins
$$\left(
\begin{array}{ccccccc}
 1 & 0 & 0 & 0 & 0 & 0 & 0 \\
 0 & -1 & 0 & 0 & 0 & 0 & 0 \\
 -2 & 0 & 1 & 0 & 0 & 0 & 0 \\
 -1 & 6 & 0 & -1 & 0 & 0 & 0 \\
 7 & 4 & -12 & 0 & 1 & 0 & 0 \\
 9 & -35 & -10 & 20 & 0 & -1 & 0 \\
 -26 & -54 & 105 & 20 & -30 & 0 & 1 \\
\end{array}
\right).$$
This is the exponential Riordan array $[g_e(x), -x]$, where $g_e(x)$ is the exponential generating function of the sequence
$$1, 0, -2, -1, 7, 9, -26, -64, 83, 407, -115,\ldots$$ which is the value of the revert polynomial for $r=0$.
The inverse Riordan array
$$(u,v)=\left(\frac{1-x}{1-x-2x^2+x^3}, \frac{x(1-x)}{1-x-2x^2+x^3}\right)^{-1}$$ also generates sequences whose Hankel transforms are given by
$$1, -2, -7, 42, 429, -7436,\ldots.$$
This is because
\begin{enumerate}
\item It is a Bell matrix
\item Its first column is the revert transform of $[x^n] \frac{1-x}{1-x-2x^2+x^3}$, which has the required Hankel transform
\item The sequences with generating function $(u, v)\cdot \frac{1}{1-rx}$ will then be INVERT$(r)$ transforms of the first column, and hence will have the same Hankel transforms.
\end{enumerate}
It is of interest to calculate the polynomial sequence
$$Q_n(r)=[x^n] \frac{1}{\left(\frac{1-x}{1-(r+1)x-(r-2)x^2+x^3}\right)^n}.$$ This begins
$$1, -r, r^2 - 4, 3(4r - 1) - r^3, r^4 - 24r^2 + 12r + 20, 5(8r^3 - 6r^2 - 20r + 7) - r^5,\ldots,$$ with a coefficient array that begins
$$\left(
\begin{array}{ccccccc}
 1 & 0 & 0 & 0 & 0 & 0 & 0 \\
 0 & -1 & 0 & 0 & 0 & 0 & 0 \\
 -4 & 0 & 1 & 0 & 0 & 0 & 0 \\
 -3 & 12 & 0 & -1 & 0 & 0 & 0 \\
 20 & 12 & -24 & 0 & 1 & 0 & 0 \\
 35 & -100 & -30 & 40 & 0 & -1 & 0 \\
 -91 & -210 & 300 & 60 & -60 & 0 & 1 \\
\end{array}
\right).$$
Again, this is an exponential Riordan array $[\tilde{g}_e, -x]$, where $\tilde{g}_e$ is the exponential generating function of the sequence
$$1,0,-4,-3,20,35,-91,\ldots,$$ which has generating function $\frac{x f'(x)}{f(x)}$ where $f(x)$ is the revert transform of $\frac{1-x}{1-x-2x^2+x^3}$.
The Hankel transform of $Q_n(r)$ evaluates to the sequence
$$1, -4, -25, 256, 4356, -123904, -5909761, 473497600, 63799687396,  \ldots.$$
Taking the absolute value and then the square root, we obtain the sequence
$$1, 2, 5, 16, 66, 352, 2431, 21760, 252586, 3803648, 74327145,\ldots.$$
For the displayed terms, this coincides with \seqnum{A005157}, which counts the number of totally symmetric plane partitions that fit in an $n \times n \times n$ box. We conjecture that these sequences are identical.

\section{The sequence $1,-2,-7,42,429,\ldots$ as a Hankel transform}
We present a table which lists sequences in the OEIS whose revert transform has a Hankel transform given by the sequence $1,-2,-7,42,429,\ldots.$ Each of these sequences can be generated by a ``small matrix'' by following the powers of that matrix. We indicate in bold the element of the matrix that is the generator.
\begin{center}
\begin{tabular} {|c|c|c|c|c|}\hline
OEIS & GF & matrix & revert transform & Hankel \\ \hline
\seqnum{A052536} & $\frac{1-x}{1-3x+x^3}$ & $\left(
\begin{array}{ccc}
 \mathbf{2} & 1 & 1 \\
 1 & 1 & 0 \\
 1 & 0 & 0 \\
\end{array}
\right)$ & $1,1,-1,-6,-8,-15$ & $1,-2,-7,42,429$ \\ \hline
\seqnum{A052547} & $\frac{1-x}{1-x-2x^2+x^3}$ & $\left(
\begin{array}{ccc}
 \mathbf{0} & 1 & 1 \\
 1 & 0 & 0 \\
 1 & 0 & 1 \\
\end{array}
\right)$ & $1,0,-2,-1,7,9$ & $1,-2,-7,42,429$ \\ \hline
\seqnum{A052941} & $\frac{1-x}{1-4x+x^2+x^3}$ & $\left(
\begin{array}{ccc}
 1 & 1 & 2 \\
 1 & 2 & 1 \\
 \mathbf{1} & 1 & 1 \\
\end{array}
\right)$ & $1,-3,7,-10,-8,111$ & $1,-2,-7,42,429$ \\ \hline
\seqnum{A077998} & $\frac{1-x}{1-2x-x^2+x^3}$ & $\left(
\begin{array}{ccc}
 \mathbf{1} & 1 & 1 \\
 1 & 1 & 0 \\
 1 & 0 & 0 \\
\end{array}
\right)$ & $1,-1,-1,4,0,-17,16 $ & $1,-2,-7,42,429$ \\ \hline
\seqnum{A052975} & $\frac{1-3x+2x^2}{1-5x+6x^2-x^3}$ & $\left(
\begin{array}{ccc}
 1 & 1 & 0 \\
 1 & \mathbf{2} & 1 \\
 0 & 1 & 2 \\
\end{array}
\right)$ & $1,-2,2,1,-5,-1,22$ & $1,-2,-7,42,429$ \\ \hline

\seqnum{A121449} & $ \frac{1-3x+2x^2}{1-4x+3x^2+x^3}$ & $\left(
\begin{array}{ccc}
 1 & 1 & 0 \\
 1 & \mathbf{1} & 1 \\
 0 & 1 & 2 \\
\end{array}
\right)$ & $1,-1,-1,2,4,-5,-20$ &  $1,-2,-7,42,429$ \\ \hline
\seqnum{A122368} & $\frac{1-3x+2x^2}{1-6x+9x^2-3x^3}$ & $\left(
\begin{array}{ccc}
 1 & 1 & 0 \\
 1 & \mathbf{3} & 1 \\
 0 & 1 & 2 \\
\end{array}
\right)$ & $1, -3, 7, -12, 12, 3, -24$ & $1,-2,-7,42,429$ \\ \hline
\seqnum{A188022} & $\frac{1+x}{1-3x^2-x^3}$ & $\begin{scriptsize}\left(
\begin{array}{cccc}
 0 & 1 & 0 & 0 \\
 1 & 0 & 1 & 0 \\
 0 & 1 & 0 & \mathbf{1} \\
 0 & 0 & 1 & 1 \\
\end{array}
\right)\end{scriptsize} $ & $1,-1,-1,6,-8,-15,84$ & $1,-2,-7,42,429$ \\ \hline
\end{tabular}
\end{center}
\begin{example} The generating function of \seqnum{A077998}, $\frac{1-x}{1-2x-x^2+x^3}$ is the generating function of the diagonal sums of the Riordan array $\left(\frac{1-x}{1-2x}, \frac{x(1-x)}{1-2x}\right)$. This sequence counts Motzkin $(n+2)$ paths with no level steps at ground level, and whose height is less than $3$. It also counts the number of compositions of $n$ if there are two kinds of part $2$. The initial column and the row sums of the Riordan array of Bell type
$$\left(\frac{1-x}{1-2x-x^2+x^3}, \frac{x(1-x)}{1-2x-x^2+x^3}\right)^{-1}$$ will have a Hankel transform that begins $1,-2,-7, 42,429,\ldots$.
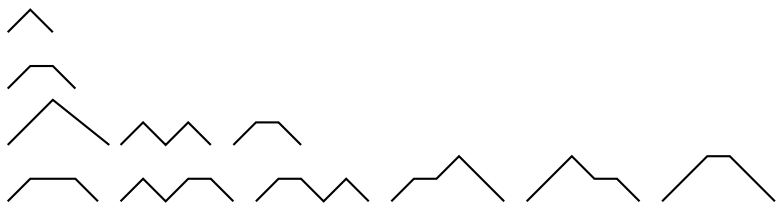
\begin{figure}
\begin{center}
\begin{tikzpicture}[scale=0.15]
\draw[black, thick] (0,0)--(2,2)--(6,2)--(8,0);
\draw[black, thick] (10,0)--(12,2)--(14,0)--(16,2)--(18,2)--(20,0);
\draw[black, thick] (22,0)--(24,2)--(24,2)--(26,2)--(28,0)--(30,2)--(32,0);
\draw[black, thick] (34,0)--(36,2)--(38,2)--(40,4)--(44,0);
\draw[black, thick] (46,0)--(50,4)--(52,2)--(54,2)--(56,0);
\draw[black, thick] (58,0)--(62,4)--(64,4)--(68,0);

\draw[black, thick] (0,15)--(2,17)--(4,15);

\draw[black, thick] (0,10)--(2,12) --(4,12)--(6,10);

\draw[black, thick] (0,5)--(4,9)--(9,5);
\draw[black, thick] (10,5)--(12,7)--(14,5)--(16,7)--(18,5);
\draw[black, thick] (20,5)--(22,7)--(24,7)--(26,5);

\filldraw[black] (0,25) circle (2pt);
\end{tikzpicture}
\caption{Motzkin paths of height less than $3$ with no flat steps at ground level}
\end{center}
\end{figure}
\end{example}
\begin{example} The log revert transform of \seqnum{A052547}, given by
$$[x^n] \left(\frac{1-x-2x^2+x^3}{1-x}\right)^n$$ has a Hankel transform that begins
$$1, -4, -25, 256, 4356, -123904,\ldots.$$
We conjecture that this is $(-1)^{\binom{n+1}{2}}$ times the square of \seqnum{A005157}, which counts the number of totally symmetric plane partitions that fit in an $n \times n \times n$ box.

\end{example}
\begin{example} We arrive at interesting results if we look at the Hankel transforms of the initial column (IC) and the row sums (RS) of the Riordan arrays defined by
$$ t_{n,k}=[x^{n-k}] \frac{1}{g(x)^n},$$ where the generating function $g(x)$ is any of the generating functions of this section. We are here using a central description of a Riordan array \cite{Central}. We document the results in the following table.
\begin{center}
\begin{tabular} {|c|c|c|c|}\hline
OEIS & GF & Hankel of IC & Hankel of RS \\\hline
\seqnum{A052536} & $\frac{1-x}{1-3x+x^3}$ & $1, -4, -25, 256, 4356, -123904$ & $1, -4, -25, 256, 4356, -123904$ \\ \hline
\seqnum{A052547} & $\frac{1-x}{1-x-2x^2+x^3}$  & $1, -4, -25, 256, 4356, -123904$ & $1, -4, -25, 256, 4356, -123904$ \\ \hline
\seqnum{A052941} & $\frac{1-x}{1-4x+x^2+x^3}$   & $1, -4, -25, 256, 4356, -123904$ & $1, -4, -25, 256, 4356, -123904$ \\ \hline
\seqnum{A077998} & $\frac{1-x}{1-2x-x^2+x^3}$  & $1, -4, -25, 256, 4356, -123904 $ & $1, -4, -25, 256, 4356, -123904$ \\ \hline
\seqnum{A052975} & $\frac{1-3x+2x^2}{1-5x+6x^2-x^3}$  & $1, -4, -33, 432, 9504, -349800$ & $1, -4, -25, 256, 4356, -123904$ \\ \hline

\seqnum{A121449} & $ \frac{1-3x+2x^2}{1-4x+3x^2+x^3}$  & $1, -4, -33, 432, 9504, -349800$ &  $1, -4, -25, 256, 4356, -123904$ \\ \hline
\seqnum{A122368} & $\frac{1-3x+2x^2}{1-6x+9x^2-3x^3}$ &  $1, -4, -33, 432, 9504, -349800$ & $1, -4, -25, 256, 4356, -123904$ \\ \hline
\seqnum{A188022} & $\frac{1+x}{1-3x^2-x^3}$  & $1, -4, -25, 256, 4356, -123904$ & $1, -4, -33, 432, 9504, -349800$ \\ \hline
\end{tabular}
\end{center}
Where the two Hankel transforms agree, this is because the initial column sequence and the row sum sequence are related by an alternating sign transform followed by an $r$-binomial transform, for suitable integer values of $r$.

We are not aware at this time of what is counted by the sequence $1, 4, 33, 432, 9504,\ldots$.
\end{example}
\begin{example} We consider the generating function
$$g_{a,c}(x)=\frac{4+4x+(a^2+1)x^2}{4+4cx+(9+4ac-3a^2)x^2-(a^3-a^2c-3a-c)x^3}.$$
This can be expressed as the continued fraction
$$\cfrac{1}{1-(a-c)x+\cfrac{2x^2}{1+\frac{a}{2}x+\cfrac{\frac{x^2}{2}}{1+\frac{a}{2}x}}}.$$
We infer from this form that the Hankel transform of the revert transform will be the Robbins numbers $A_{n+1}$. For instance, for $(a,c)=(0,0)$ we obtain the generating function
$$g_{0,0}(x)=\frac{4+x^2}{4+9x^2}.$$ The revert transform of the expansion of this generating function begins
$$1, 0, 2, 0, \frac{15}{2}, 0,\frac{273}{8}, 0,\frac{5471}{32} , 0, \frac{116193}{128},\ldots.$$
As expected, its Hankel transform is $1,2,7,42,\ldots$.
The generating function $$g_{0,0}(2x)=\frac{1+x^2}{1+9x^2}$$ leads to the revert transform that begins
$$1, 0, 8, 0, 120, 0, 2184, 0, 43768, 0, 929544,\ldots.$$
This sequence will then have its Hankel transform given by $[4^{\binom{n+1}{2}}](1,2,7,42,\ldots]$.

In like manner, taking $(a,c)=(1,-1)$, we get the generating function
$$g_{1,-1}(x)=\frac{2+2x+x^2}{2-2x+x^2}.$$
The revert transform of the expansion of $g_{1,-1}(x)$ begins
$$1, -2, 6, -21, 80, - \frac{643}{2},\frac{2681}{2}, - \frac{22967}{4}, 25104, - \frac{892409}{8}, \ldots.$$
Its Hankel transform is $1,2,7,42,\ldots$.
The generating function $$g_{1,-1}(2x)=\frac{1+2x+2x^2}{1-2x+2x^2}$$ will have an expansion whose revert transform begins
$$1, -4, 24, -168, 1280, -10288, 85792, -734944, 6426624, -57114176, \ldots.$$
This sequence will then have its Hankel transform given by $[4^{\binom{n+1}{2}}](1,2,7,42,\ldots]$.
\end{example}
\begin{example} We consider the generating function
$$f_{b,c}(x)=\frac{1+\sqrt{4b-1}x+bx^2}{1+cx+(c\sqrt{4b-1}-3b+3)x^2+\frac{1}{4}(bc+(1-b)\sqrt{4b-1})x^3}.$$
This generating function may be expressed as the continued fraction
$$\cfrac{1}{1-(\sqrt{4b-1}-c)x+\cfrac{2x^2}{1+\frac{\sqrt{4b-1}}{2}x+\cfrac{\frac{x^2}{4}}{1+\frac{\sqrt{4b-1}}{2} x}}}.$$
Again, this form assures us that the revert transform of the expansion of this generating function will have a Hankel transform given by the Robbins numbers $A_{n+1}$.
This is particularly the case \cite{Gessel} for
$$f_{1,0}(x)=1+\sqrt{3}x+x^2.$$
Other interesting generating functions are
\begin{align*}
f_{0,0}(x)&=\frac{1+ix}{1+3x^2+ix^3}\\
f_{1,1}(x)&=\frac{1+\sqrt{3}x+x^2}{1+x+\sqrt{3}x+x^3}\\
f_{1,-1}(x)&=\frac{1+\sqrt{3}x+x^2}{1-x-\sqrt{3}-x^3}\\
f_{1,-3}(x)&=\frac{1+\sqrt{3}x+x^2}{1-3x-3\sqrt{3}x^2-3x^3}.\end{align*}
We take a closer look at $f_{0,0}$, which can be expressed as
$$f_{0,0}(x)=\frac{1+ix}{1+3x^2+ix^3}=\frac{1+3x^2+4x^4}{1+6x^2+9x^4+x^6}+\frac{i(1+2x^2)x}{1+6x^2+9x^4+x^6}.$$
The expansion of $f_{0,0}(x)$, then begins
$$1, i, -3, - 4i, 10, 15i, -34, - 55i, 117, 199i, -406,\ldots,$$ or
$$1, 0, -3, 0, 10, 0, -34, 0, 117, 0 ,\ldots + i(0, 1, 0, -4, 0, 15, 0, -55, 0, \ldots).$$
The summands here are aerations of alternating sign versions of \seqnum{A094832} and \seqnum{A094833}.
The corresponding real positive sequence \seqnum{A188022} has generating function $\frac{1+x}{1-3x^2-x^3}$ and begins
$$1, 1, 3, 4, 10, 15, 34, 55, 117, 199, 406, 714,\ldots.$$
\end{example}
\begin{example} \textbf{The path graphs $P_3$ and $P_6$ and the Robbins numbers.}
We consider the path graph $P_3$,  to which a loop is added at the end. This modified path graph will then have its adjacency matrix given by
$$\left(
\begin{array}{ccc}
 0 & 1 & 0 \\
 1 & 0 & 1 \\
 0 & 1 & 1 \\
\end{array}
\right).$$
Taking the powers of this matrix and following the $(2,2)$-element, we obtain the sequence \seqnum{A052547} that begin
$$ 1, 0, 2, 1, 5, 5, 14, 19, 42, 66, 131,\ldots,$$ with generating function
$$\frac{1-x}{1-x-2x^2+x^3}.$$
Complexifying this as in the last example, we consider the sequence that begins
$$ 1, 0, -2, -i, 5, 5i, -14, - 19i, 42, 66i, -131,\ldots.$$
This sequence has its generating function given by
$$\frac{i+x}{i+x+2ix^2+x^3}=\cfrac{1}{1+\cfrac{x^2}{1-\frac{i}{2}x+\cfrac{\frac{x^2}{4}}{1-\frac{i}{2}x}}}.$$
Thus the expansion of the revert transform of the above sequence, which begins
$$1, 0, 2, i, 7, 9i, 26, 64i, 83, 407i, 115,\ldots,$$ will have the Robbins numbers $A_{n+1}$ as Hankel transform.

We can express the generating function above as
$$\frac{i+x}{i+x+2ix^2+x^3}=\frac{1+3x^2+x^4}{1+5x^2+6x^4+x^6}-\frac{ix^3}{1+5x^2+6x^4+x^6}.$$
We consider the adjacency matrix of the path graph $P_6$. This is given by
$$\left(
\begin{array}{cccccc}
 0 & 1 & 0 & 0 & 0 & 0 \\
 1 & 0 & 1 & 0 & 0 & 0 \\
 0 & 1 & 0 & 1 & 0 & 0 \\
 0 & 0 & 1 & 0 & 1 & 0 \\
 0 & 0 & 0 & 1 & 0 & 1 \\
 0 & 0 & 0 & 0 & 1 & 0 \\
\end{array}
\right).$$
Taking powers and following the $(1,2)$ position, we find that paths from vertex $1$ to vertex $2$ are counted by the sequence
$$0, 1, 0, 2, 0, 5, 0, 14, 0, 42, 0, 131, 0,\ldots,$$ with generating function
$$\frac{x(1-3x^2+x^4)}{1-5x^2+6x^4-x^6}.$$ This is (essentially) an aeration of \seqnum{A080937}.

If we now follow the $(1,6)$ position, we find that paths from vertex $1$ to vertex $6$ are counted by the sequence that begins
$$0, 0, 0, 0, 0, 1, 0, 5, 0, 19, 0, 66, 0, 221, 0, 728, 0, 2380, 0, 7753, 0,\ldots.$$
This is essentially an aeration of \seqnum{A005021}. The aerated sequence has generating function
$\frac{x^5}{1-5x^2+6x^4-x^6}$.
\end{example}

\section{The sequence $1,1,2,6,33,286,\ldots$ as a Hankel transform}
The sequence \seqnum{A005161}, which begins
$$1, 1, 1, 2, 6, 33, 286, 4420, 109820, 4799134, 340879665,\ldots,$$
counts the number of alternating sign $(2n+1) \times (2n+1)$ matrices that are symmetric with respect to both horizontal and vertical axes.

In this section, we present a table that documents sequences and their revert transforms, where the listed revert transforms have as their Hankel transform the sequence $1,1,2,6,33,286,\ldots$. Where a sequence is not listed in the OEIS, we give some starting terms. The superscript $^a$ indicates an aerated sequence.
\begin{center}
\begin{tabular} {|c|c|c|c|c|}\hline
OEIS & GF  & revert transform & OEIS & Hankel \\ \hline
$1,0,-1,0,0,0,0$ & $1-x^2$ & $1,0,1,0,3,0,12$ & \seqnum{A001764}$^a$ & $1,1,2,6,33,286$ \\ \hline
$(-1)^n$\seqnum{A080956} & $\frac{1+2x}{(1+x)^3}$ & $1,1,2,3,7,12$ & \seqnum{A047749}$_{n+1}$ &  $1,1,2,6,33,286$ \\ \hline
\seqnum{A122100} & $\frac{1-4x+3x^2}{1-3x+x^3}$ & $1,2,2,6,22,90$ & \seqnum{A049126} & $1,1,2,6,33,286$ \\ \hline
\seqnum{A104769}$(n+2)$ & $\frac{1-x^2}{1+x-x^2}$ & $1,1,2,4,10,26$  & \seqnum{A049130} & $1,1,2,6,33,286$ \\ \hline
$1,2,3,4,4,2$ & $\frac{1-x^2}{1-2x+2x^3}$ & $1,-2,5,14,43$ & \seqnum{A088927} & $1,1,2,6,33,286$ \\ \hline
\seqnum{A080956} & $\frac{1-2x}{(1-x)^3}$ & $1,-1,2,-3,7,-12$ & \seqnum{A134565} & $1,1,2,6,33,286$ \\ \hline
$1,2,3,5,7,12$ & $\frac{1+2x}{1-3x^2+x^3}$ & $1,-2,5,-15,52$ &  & $1,1,2,6,33,286$ \\ \hline
$1,-3,8,-21,54$ & $\frac{1-x^2}{1+3x-3x^3}$ & $1,3,10,36,138$ &  & $1,1,2,6,33,286$ \\ \hline
$1,-1,0,0,-1,-1,-2$ & $\frac{1-2x}{1-x-x^2-x^3}$ & $1,1,2,5,15,50$ &  & $1,1,2,6,33,286$ \\ \hline
$1,-2,3,-5,7,-12$ & $\frac{1-2x}{1-3x^2-x^3}$ & $1,2,5,15,52$ &  & $1,1,2,6,33,286$ \\ \hline
\seqnum{A154272} & $1+x^2$ & $1, 0, -1, 0, 3, 0, -12$ & \seqnum{A001764}$^b$ & $1, -1, -2, 6, 33, -286$ \\\hline
\end{tabular}
\end{center}
We note that the Hankel transform of \seqnum{A047749} begins
$$1, 0, -1, 0, 9, 0, -676, 0, 417316, 0, -2105433225,\ldots.$$
Taking the square root of the absolute value of these terms, we obtain
$$1, 0, 1, 0, 3, 0, 26, 0, 646, 0, 45885,\ldots,$$ an aerated version of  \seqnum{A005156}.

In order to construct a symmetric matrix whose principal minor sequence is given by $1,1,2,6,33,286,\ldots$, we take $f(x)=1-x^2$ and we expand the bivariate generating function
$$\frac{f(x)}{1-x} \frac{x-y}{xf(x)-yf(y)} \frac{f(y)}{1-y}=\frac{(1+x)(1+y)}{1-x^2-xy-y^2}$$ to get the number square that begins
$$\left(
\begin{array}{ccccccc}
 1 & 1 & 1 & 1 & 1 & 1 & 1 \\
 1 & 2 & 2 & 3 & 3 & 4 & 4 \\
 1 & 2 & 4 & 5 & 8 & 9 & 13 \\
 1 & 3 & 5 & 10 & 13 & 22 & 26 \\
 1 & 3 & 8 & 13 & 26 & 35 & 61 \\
 1 & 4 & 9 & 22 & 35 & 70 & 96 \\
 1 & 4 & 13 & 26 & 61 & 96 & 192 \\
\end{array}
\right).$$
This can be built up from the embedded triangle that begins
$$\left(
\begin{array}{cccccccc}
 1 & 0 & 0 & 0 & 0 & 0 & 0 & 0 \\
 2 & 1 & 0 & 0 & 0 & 0 & 0 & 0 \\
 4 & 2 & 1 & 0 & 0 & 0 & 0 & 0 \\
 10 & 5 & 3 & 1 & 0 & 0 & 0 & 0 \\
 26 & 13 & 8 & 3 & 1 & 0 & 0 & 0 \\
 70 & 35 & 22 & 9 & 4 & 1 & 0 & 0 \\
 192 & 96 & 61 & 26 & 13 & 4 & 1 & 0 \\
 534 & 267 & 171 & 75 & 40 & 14 & 5 & 1 \\
\end{array}
\right)$$ by the symmetrization process already described. This matrix constitutes an interleaving of the two Riordan arrays
$$\left(\frac{1-x}{1+x}, \frac{x}{1+x+x^2}\right)^{-1}$$ and
$$\left(\frac{(1-x)^2}{1-x^3}, \frac{x}{1+x+x^2}\right)^{-1}.$$
These begin, respectively,
$$\left(
\begin{array}{cccccc}
 1 & 0 & 0 & 0 & 0 & 0 \\
 2 & 1 & 0 & 0 & 0 & 0 \\
 4 & 3 & 1 & 0 & 0 & 0 \\
 10 & 8 & 4 & 1 & 0 & 0 \\
 26 & 22 & 13 & 5 & 1 & 0 \\
 70 & 61 & 40 & 19 & 6 & 1 \\
\end{array}
\right), \left(
\begin{array}{cccccc}
 1 & 0 & 0 & 0 & 0 & 0 \\
 2 & 1 & 0 & 0 & 0 & 0 \\
 5 & 3 & 1 & 0 & 0 & 0 \\
 13 & 9 & 4 & 1 & 0 & 0 \\
 35 & 26 & 14 & 5 & 1 & 0 \\
 96 & 75 & 45 & 20 & 6 & 1 \\
\end{array}
\right).$$
The production matrix of the ``embedded'' triangle begins
$$\left(
\begin{array}{cccccccc}
 2 & 1 & 0 & 0 & 0 & 0 & 0 & 0 \\
 0 & 0 & 1 & 0 & 0 & 0 & 0 & 0 \\
 2 & 1 & 1 & 1 & 0 & 0 & 0 & 0 \\
 0 & 0 & 0 & 0 & 1 & 0 & 0 & 0 \\
 2 & 1 & 1 & 1 & 1 & 1 & 0 & 0 \\
 0 & 0 & 0 & 0 & 0 & 0 & 1 & 0 \\
 2 & 1 & 1 & 1 & 1 & 1 & 1 & 1 \\
 0 & 0 & 0 & 0 & 0 & 0 & 0 & 0 \\
\end{array}
\right).$$
The elements of the inverse of the ``embedded triangle'' are then the coefficients of the characteristic polynomials of the principal minor matrices of this production array. Thus this production matrix provides a direct link to the sequence $1,1,2,6,33,286,\ldots$.

\section{The sequence  $1,3,26,646,\ldots$ as a Hankel transform}
The sequence \seqnum{A005156}, which begins
$$1, 1, 3, 26, 646, 45885, 9304650, 5382618660, 8878734657276,\ldots,$$
counts the number of  alternating sign $2n+1 \times 2n+1$ matrices that are symmetric about the vertical axis.

In this section we present a table of sequences whose revert transforms have the sequence $1,3,26,646,\ldots$ as their Hankel transform.
\begin{center}
\begin{tabular} {|c|c|c|c|c|}\hline
OEIS & GF  & revert transform & OEIS & Hankel \\ \hline
$(-1)^n$\seqnum{A130713} & $(1-x)^2$ & $1,2,7,30,143$ & \seqnum{A006013} & $1,3,26,646,\ldots$ \\ \hline
\seqnum{A130713} & $(1+x)^2$ & $1,-2,7,-30,143$ & $(-1)^n$\seqnum{A006013} &$1,3,26,646,\ldots$ \\ \hline
$(-1)^n$\seqnum{A000217}$_{n+1}$ & $\frac{1}{(1+x)^3}$ & $1,3,12,55,273$ & \seqnum{A001764}$_{n+1}$ & $1,3,26,646,\ldots$ \\ \hline
$(-1)^n$\seqnum{A200715}$_{n+3}$ & $\frac{1}{1+x+3x^2+x^3}$ & $1,1,4,11,41,146$ & \seqnum{A030981} & $1,3,26,646,\ldots$ \\ \hline
\seqnum{A127896} & $\frac{1}{1+2x+3x^2+x^3}$ & $1, 2, 7, 27, 114, 507$ & \seqnum{A127897} & $1,3,26,646,\ldots$ \\ \hline
$1,-1,-2,-2,-1$ & $\frac{(1-2x)^2}{(1-x)^3}$ & $ 1, 1, 4, 17, 81, 412, $ & \seqnum{A121545} & $1,3,26,646,\ldots$ \\ \hline
 $1,-4,13,-38,104$ & $\frac{(1+x)^2}{(1+2x)^3}$ & $1, 4, 19, 98, 531, 2974$ & \seqnum{A047099} & $1,3,26,646,\ldots$ \\ \hline
\seqnum{A215404} & $\frac{1}{1-4x+3x^2+x^3}$ & $1,-4,19,-99,546$ &  & $1,3,26,646,\ldots$ \\ \hline
$1,0,-3,-1,9,6$ & $\frac{1}{1+3x^2+x^3}$ & $1,0,3,1,18,15$ & \seqnum{A120984} & $1,3,26,646,\ldots$ \\ \hline
\seqnum{A339850} & $\frac{(1+x)^2}{1-2x-4x^2-2x^3}$ & $1,-4,19,-104,631$ & & $1,3,26,646,\ldots$ \\ \hline
\seqnum{A077954}$_{n+1}$ & $\frac{(1-x)^2}{1-x+2x^2-x^3}$ & $1,1,14,58,252$ & & $1,3,26,646,\ldots$ \\ \hline
\end{tabular}
\end{center}
We note that the Hankel transforms of the log  revert transform of both $\frac{1}{1+3x^2+x^3}$ and $\frac{1}{1-4x+3x^2+x^3}$ are given by $2^n[1,2,11,170,\ldots]$.

In order to find a symmetric matrix whose principal minor sequence is $1,3,26,\ldots$, we start with the generating function $f(x)=\frac{1}{(1+x)^3}$, and we form the matrix whose bivariate generating function is given by
$$\frac{f(x)}{1-x} \frac{x-y}{xf(x)-yf(y)} \frac{f(y)}{1-y}=\frac{1}{(1-y)(1-y)(1+(y-3)xy+x^2y)}.$$ This matrix begins
$$\left(
\begin{array}{ccccccc}
 1 & 1 & 1 & 1 & 1 & 1 & 1 \\
 1 & 4 & 5 & 5 & 5 & 5 & 5 \\
 1 & 5 & 15 & 21 & 22 & 22 & 22 \\
 1 & 5 & 21 & 56 & 84 & 93 & 94 \\
 1 & 5 & 22 & 84 & 211 & 331 & 386 \\
 1 & 5 & 22 & 93 & 331 & 802 & 1298 \\
 1 & 5 & 22 & 94 & 386 & 1298 & 3069 \\
\end{array}
\right).$$ By construction, the principal minor sequence of this matrix is given by $1,3,26,646,\ldots$.
In this case, the ``embedded triangle'' is not a Riordan array.

We can also parameterize our approach as in the next example.
\begin{example}
We consider the generating function $\frac{1}{1+3tx+3x^2+x^3}$, a parameterised version (in the linear part) of $\frac{1}{(1+x)^3}$. The revert transform begins
$$1, 3t, 3(3t^2 + 1), 27t^3 + 27t + 1, 3(27t^4 + 54t^2 + 4t + 6),\ldots.$$ This polynomial sequence has a coefficient array that begins
$$\left(
\begin{array}{cccccc}
 1 & 0 & 0 & 0 & 0 & 0 \\
 0 & 3 & 0 & 0 & 0 & 0 \\
 3 & 0 & 9 & 0 & 0 & 0 \\
 1 & 27 & 0 & 27 & 0 & 0 \\
 18 & 12 & 162 & 0 & 81 & 0 \\
 15 & 270 & 90 & 810 & 0 & 243 \\
\end{array}
\right).$$ This is \seqnum{A120981}, which counts the number of ternary trees with $n$ edges and having $k$ vertices of out-degree $1$ (E. Deutsch). Dividing the columns by $3^n$, we get the matrix that begins
$$\left(
\begin{array}{cccccc}
 1 & 0 & 0 & 0 & 0 & 0 \\
 0 & 1 & 0 & 0 & 0 & 0 \\
 3 & 0 & 1 & 0 & 0 & 0 \\
 1 & 9 & 0 & 1 & 0 & 0 \\
 18 & 4 & 18 & 0 & 1 & 0 \\
 15 & 90 & 10 & 30 & 0 & 1 \\
\end{array}
\right).$$ This is the exponential Riordan array $[g_e(x), x]$ where $g(x)$ is the exponential generating function of $1,0,3,1,18,15,\ldots$ or \seqnum{A120984}. This last sequence is the revert transform of the sequence with generating function $\frac{1}{1+3x^2+x^3}$ corresponding to $t=0$. This sequence, and the row sums of the matrix (\seqnum{A030981}), will have Hankel transforms $1,3,26,646,\ldots$. Indeed, all sequences with generating function $g_e(x)e^{rx}$ will have this property. This matrix is the coefficient matrix of the polynomials in $r$ defined by
$$ \frac{1}{n+1} [x^n] (1+r x + 3 x^2+ x^3)^{n+1}.$$ It is the inversion of the Riordan array
$$\left(\frac{1}{1+3x^2+3x^3}, \frac{-x}{1+3x^2+3x^3}\right).$$
\end{example}
\begin{example} The sequence \seqnum{A098746} has generating function
$$1+\frac{x t(x)}{1-x t(x)},$$  where $t(x)$ is the generating function of the ternary numbers. The Hankel transform of \seqnum{A098746} is given by
$$1, 1, 3, 26, 646, 45885, 9304650, 5382618660, 8878734657276, 41748486581283118, \ldots.$$
The sequence \seqnum{A098746} is the initial column in the Riordan array
$$\left(1-x+x^2, x(1-x)^2\right)^{-1}.$$
The row sums of this Riordan array are given by the shifted sequence \seqnum{A098746}$(n+1)$, and have a Hankel transform which begins
$$1, 2, 11, 170, 7429, 920460,\ldots.$$

The row sums of the Riordan array
$$\left(\frac{1-x}{(1+x)^2}, \frac{x}{(1+x)^3}\right)^{-1},$$ which begin
$$1, 4, 20, 108, 608, 3516, 20724, 123920, 749408, 4573788, 28127996,\ldots,$$
have a Hankel transform which begins
$$4^n[1, 1, 3, 26, 646, 45885,\ldots].$$
\end{example}
\begin{example} We consider the Riordan array of Bell type $((1+x)^2, x(1+x)^2)$. The initial column of the inverse matrix $((1+x)^2, x(1+x)^2)^{-1}$ is the revert transform of the expansion of $(1+x)^2$. The row sums of the inverse matrix is the invert transform of the initial column, and hence has the same Hankel transform, namely $1,3,26,646,\ldots$. The row sums of $((1+x)^2, x(1+x)^2)$ are given by \seqnum{A002478}$(n+1)$, where \seqnum{A002478} counts the number of ways to tile a $3 \times n$3 region with $1 \times 1$, $2 \times 2$, and $3 \times 3$ tiles. The inversion of the Riordan array $((1+x)^2, x(1+x)^2)$ is given by the exponential Riordan array $[g_e(x),-x]$ which begins
$$\left(
\begin{array}{cccccc}
 1 & 0 & 0 & 0 & 0 & 0 \\
 -2 & -1 & 0 & 0 & 0 & 0 \\
 7 & 4 & 1 & 0 & 0 & 0 \\
 -30 & -21 & -6 & -1 & 0 & 0 \\
 143 & 120 & 42 & 8 & 1 & 0 \\
 -728 & -715 & -300 & -70 & -10 & -1 \\
\end{array}
\right),$$ where $g_e(x)$ is the exponential generating function of $1,-2,7,-30,\ldots$, or $(-1)^n \frac{1}{n+1} \binom{3n+1}{n+1}$. All sequences with generating function $[g_e(x), x]\cdot e^{rx}$ will then yield sequences whose Hankel transform is given by $1,3,26,646,\ldots$.
\end{example}
\begin{example} We consider the family of polynomials given by
$$\frac{1}{n+1} [x^n] \frac{1}{(1+rx+x^2)^{n+1}},$$ which begin
$$1, -r, 2r^2-1, 5r(1-r^2), 14r^4-21r^2+3,\ldots.$$
These have a coefficient array with generating function
$$\frac{2 \sqrt{r^2-3}\sin\left(\frac{1}{3}\arcsin\left(\frac{r(2r^2-9)-27x}{2(r^2-3)^{\frac{3}{2}}}\right)\right)}{3x}-\frac{r}{3x}.$$
This array begins
$$\left(
\begin{array}{ccccccc}
 1 & 0 & 0 & 0 & 0 & 0 & 0 \\
 0 & -1 & 0 & 0 & 0 & 0 & 0 \\
 -1 & 0 & 2 & 0 & 0 & 0 & 0 \\
 0 & 5 & 0 & -5 & 0 & 0 & 0 \\
 3 & 0 & -21 & 0 & 14 & 0 & 0 \\
 0 & -28 & 0 & 84 & 0 & -42 & 0 \\
 -12 & 0 & 180 & 0 & -330 & 0 & 132 \\
\end{array}
\right).$$
For $r= \pm 2$ we get $(\mp 1)^n \frac{1}{n+1}\binom{3n+1}{n}$, both with the Hankel transform $1,3,26,646,\ldots$.
For $r=0$, we get the sequence  $1,0,-1,0,3,0,-12,\ldots$ with Hankel transform $$1,-1,-2,6,33,-286,\ldots.$$
\end{example}
\begin{example} The principal minor sequence of the symmetrization of the Riordan array
$$\left(\frac{1+2x}{1+x+x^2}, \frac{x}{1+x}\right)$$ is given by
$$1,0,-1^2,0,3^2,0,-26^2,0,646^2,0,\ldots.$$
\end{example}

\section{The sequence $1,2,11,170,7429,\ldots$ as a Hankel transform}
The sequence \seqnum{A051255} which begins
$$1, 1, 2, 11, 170, 7429, 920460, 323801820, 323674802088,\ldots,$$
counts the number of cyclically symmetric transpose complement plane partitions in a $2n \times 2n \times 2n$ box.
The sequence $1,2,11,170,7429,\ldots$ is given by, for instance, the Hankel transform of the ternary $\frac{1}{2n+1}\binom{3n}{n}$ \seqnum{A007614}. The ternary numbers are the revert transform of the sequence $1,-1,-1,-2,-5,-14,\ldots$ with generating sequence
$$\frac{1}{c(x)}=1-x c(x).$$
This generating function is not a rational function, as has been the case with the sequences examined heretofore. It turns out that its binomial and invert transforms which would also lead, through reversion, to the same Hankel transform, have been little studied. Our table for $1,2,11,170,\ldots$ linking to entries in the OEIS is therefore short.
\begin{center}
\begin{tabular}{|c|c|c|}\hline
sequence & revert transform & Hankel transform\\ \hline
\seqnum{A115140} & \seqnum{A007614} & $1,2,11,170, 7429,\ldots$ \\ \hline
$1,-3,7,-18,43,-109$ & \seqnum{A186185}$_{n+1}$ & $1,2,11,170,7249,\ldots$ \\ \hline
$1,-2,2,-5,2,-18$ & \seqnum{A188687} & $1,2,11,170,7249,\ldots$ \\ \hline
$1,-5,23,-102,443$ & \seqnum{A305573} & $1,2,11,170,7249,\ldots$ \\ \hline
$1,-2,0,2,4,2,-12$ &\seqnum{A047098} & $2^n[1,2,11,170,7249,\ldots]$ \\ \hline
$1,-3,3,6,-9,-42$ & \seqnum{A005809} & $3^n[1,2,11,170,7429,\ldots]$ \\ \hline
\seqnum{A099325} & $1,-3,11,-46,211,-1035$ & $1,2,11,170,7249,\ldots$ \\ \hline
$1,,-2,2,-3,2,-5$ & \seqnum{A098746}$_{n+1}$ & $1,2,11,170,7249,\ldots$ \\ \hline
 & \seqnum{A007226} & $2,11,170,7249,\ldots$ \\ \hline
\end{tabular}
\end{center}
\begin{example} We follow the example of the last section. Thus we consider the log revert transform
$[x^n](1+tx+3x+x^3)$ of $\frac{1}{1+tx+3x^2+x^3}$. This produces the polynomial sequence that begins
$$1, t, t^2 + 6, t^3 + 18t + 3, t^4 + 36t^2 + 12t + 54,\ldots,$$ with coefficient array that begins
$$\left(
\begin{array}{ccccccc}
 1 & 0 & 0 & 0 & 0 & 0 & 0 \\
 0 & 1 & 0 & 0 & 0 & 0 & 0 \\
 6 & 0 & 1 & 0 & 0 & 0 & 0 \\
 3 & 18 & 0 & 1 & 0 & 0 & 0 \\
 54 & 12 & 36 & 0 & 1 & 0 & 0 \\
 60 & 270 & 30 & 60 & 0 & 1 & 0 \\
 555 & 360 & 810 & 60 & 90 & 0 & 1 \\
\end{array}
\right).$$
This is the exponential Riordan array $[g_e(x), x]$ where $g_e(x)$ is the exponential generating function of the sequence $1,0,6,3,54,60,\ldots$ which corresponds to $t=0$. For all values of $t$, the polynomial sequence has a Hankel transform given by $3^n[1,2,11,170,\ldots]$.
\end{example}

\section{Fibonacci numbers, Catalan numbers, and the Robbins numbers}
In this section, we propose a conjecture concerning the Fibonacci numbers, the Catalan numbers, and the Robbins numbers. We recall that
$$c(x)= \frac{1-\sqrt{1-4x}}{2x}$$ is the generating function of the Catalan numbers $C_n=\frac{1}{n+1}\binom{2n}{n}$ \seqnum{A000108}. The generating function $c(-x^2)$, which expands to give the aeration of the alternating sign Catalan numbers
$$1,0, -1,0,2,0,-5,0,14,0,-42,0,\ldots$$ has the continued fraction form
$$c(-x^2)=\frac{1}{1+\cfrac{x^2}{1+\cfrac{x^2}{1+\cfrac{x^2}{1+\cdots}}}}.$$
Using the properties of the binomial transform \cite{CFT}, the generating function of the inverse binomial transform of this generating function will have the form
$$\frac{1}{1+x}c\left(-\left(\frac{x}{1+x}\right)^2\right)=\cfrac{1}{1+x+\cfrac{x^2}{1+x+\cfrac{x^2}{1+x+\cfrac{x^2}{1+x+\cdots}}}}.$$ This expands to give the revert transform of the Fibonacci numbers $F_{n+1}$ \seqnum{A007440} which begins
$$1,-1,0,2,-3,-1,11,-15,-13,77,\ldots.$$
This then has generating function
$$F(x)=\frac{\sqrt{1+2x+5x^2}-x-1}{2x^2}.$$
We now form the generating function
\begin{align*}\tilde{A}(x)&=\cfrac{1}{1-\cfrac{x^2}{1+x^2 F(x)}}\\
&=\frac{1-x \sqrt{1+2x+5x^2}}{1-x-2x^2+\sqrt{1+2x+5x^2}}\\
&=\frac{2+3x-x^2+x \sqrt{1+2x+5x^2}}{2(1+2x-x^2-x^3)}.\end{align*}
Alternatively, we may write this as the continued fraction
$$\tilde{A}(x)=\cfrac{x}{1-\cfrac{x^2}{1+\cfrac{x^2}{1+x+\cfrac{x^2}{1+x+\cfrac{x^2}{1+x+\cdots}}}}}.$$
The revert transform of $\tilde{A}(x)$ then expands to give the sequence that begins
$$1,0,-1,0,3,-1,-12,11,51,-89,-204,628,646,\ldots.$$
The Hankel transform of this sequence begins
$$1,-1,-2,7,42,-429,-7436,218348,10850216, -911835460,\ldots.$$
We thus arrive at the following conjecture.
\begin{conjecture} The Robbins numbers
$$1,1,2,7,42,429, 7436,\ldots$$ are given by the Hankel transform of the revert transform
$$1,0,1,0,3,-i,12,-11i,51,-89i,204,-628i,646,\ldots$$
of the sequence whose generating function is given by
$$\tilde{A}(ix)=\cfrac{1}{1+\cfrac{x^2}{1-\cfrac{x^2}{1+ix-\cfrac{x^2}{1+ix-\cfrac{x^2}{1+ix-\cdots}}}}}.$$
\end{conjecture}

\section{Riordan arrays and the Robbins numbers}
We now wish to start with the generating function $\tilde{A}(x)$ of the last section to arrive at a Riordan array whose symmetrization yields the sequence $1,-1,-2,7,42,\ldots$. A  modification in the symmetrization process will then yield a matrix whose principal minor sequence gives the Robbins numbers (subject to the conjecture of the last section being valid).
For this, we let
$$g(x)=\frac{x}{1-x-x^2},$$ and
$$\gamma(x)=\frac{1-x-x^2}{1-2x-x^2}.$$
We use these generating functions to transform the generating function of the Hankel matrix for $\tilde{A}(x)$ in such a way that the principal minor sequences of the resulting matrix will yield the Hankel transform. Thus we calculate
$$\frac{1}{\gamma(x)(1-x)}\frac{1}{\gamma(y)(1-y)}\frac{\tilde{A}(g(x))}{1-g(x)} \frac{\tilde{A}(g(y)}{1-g(y)}\frac{g(x)-g(y)}{g(x) \tilde{A}(g(x))- g(y)\tilde{A}(g(x))}.$$
This results in the following generating function.
$$\frac{1+xy}{1-x-y+3xy-xy^2-x^2y+x^2y^2}.$$
The homogeneous form of this generating function ensures that it is the bivariate generating function of a symmetric matrix, whose principal minor sequence will, by construction, be $1,-1,-2,7,42,\ldots$.
In fact, this matrix, which begins
$$\left(
\begin{array}{ccccccc}
 1 & 1 & 1 & 1 & 1 & 1 & 1 \\
 1 & 0 & -1 & -2 & -3 & -4 & -5 \\
 1 & -1 & -1 & 0 & 2 & 5 & 9 \\
 1 & -2 & 0 & 1 & 1 & -1 & -6 \\
 1 & -3 & 2 & 1 & 0 & -1 & 0 \\
 1 & -4 & 5 & -1 & -1 & -1 & 0 \\
 1 & -5 & 9 & -6 & 0 & 0 & 1 \\
\end{array}
\right),$$
is the symmetrization of the Riordan array $\left(\frac{1+x}{1+x+x^2}, \frac{x}{1+x}\right)$.
We now modify the symmetrization process to produce a single \emph{amalgamated} matrix from two Riordan arrays.
Thus given two Riordan arrays $A=(a_{n,k})$ and $B=(b_{n,k})$  we form the matrix $A \amalg B=(t_{n,k})$ where
$$t_{n,k}=\begin{cases} a_{n,n-k}\quad \text{if\,}k \le n,\\
                        b_{k,k-n} \quad \text{otherwise}.
           \end{cases}$$
Using this notation, the symmetrization of the Riordan array $A$ is just $A  \amalg A$.
We then have the following result.
\begin{conjecture} The principal minor sequence of the amalgamation
$$\left(\frac{1-x}{1-x+x^2}, \frac{x}{1-x}\right) \amalg \left(\frac{1-x}{1-x+x^2}, \frac{-x}{1-x}\right)$$
is the sequence of Robbins numbers.
\end{conjecture}
This results from the fact that multiplying the elements of an $n \times n$ matrix by $(-1)^k$ produces a factor of $(-1)^{\binom{n}{2}}$ for the determinant. The new matrix begins
$$\left(
\begin{array}{cccccccc}
 1 & -1 & 1 & -1 & 1 & -1 & 1 & -1 \\
 1 & 0 & -1 & 2 & -3 & 4 & -5 & 6 \\
 1 & 1 & -1 & 0 & 2 & -5 & 9 & -14 \\
 1 & 2 & 0 & -1 & 1 & 1 & -6 & 15 \\
 1 & 3 & 2 & -1 & 0 & 1 & 0 & -6 \\
 1 & 4 & 5 & 1 & -1 & 1 & 0 & 0 \\
 1 & 5 & 9 & 6 & 0 & 0 & 1 & -1 \\
 1 & 6 & 14 & 15 & 6 & 0 & 1 & 0 \\
\end{array}
\right).$$
Multiplying this matrix on the right by the transpose of the Riordan array $\left(\frac{1+x}{1-x}, x\right)$ leads us to the matrix that begins
$$\left(
\begin{array}{ccccccc}
 1 & 1 & 1 & 1 & 1 & 1 & 1 \\
 1 & 2 & 1 & 2 & 1 & 2 & 1 \\
 1 & 3 & 3 & 2 & 4 & 1 & 5 \\
 1 & 4 & 6 & 5 & 5 & 7 & 2 \\
 1 & 5 & 10 & 11 & 10 & 11 & 12 \\
 1 & 6 & 15 & 21 & 21 & 21 & 22 \\
 1 & 7 & 21 & 36 & 42 & 42 & 43 \\
\end{array}
\right).$$
This matrix has been encountered previously in Example \ref{Ex}.

\section{Conclusions} This article has provided evidence that many of the determinantal formulas for sequences involving the enumeration of types of plane partitions and alternating sign matrices can be linked to two related ideas: that of the principal minor sequences of symmetric matrices, on the one hand, and that of Hankel transforms of the revert transform of simpler sequences connected with the enumeration of simpler objects (for instance, tilings and height-limited Motzkin paths). It has been found useful also to invoke the theory of Riordan arrays and continued fractions to find relationships that cast further light on the richness of this area of study.

\section{Acknowledgements}
This article makes reference to a number of sequences, many of which are to be found in the OEIS. The present author has found this encyclopedia to be an invaluable tool for this work.

The code for the image of the plane partition is that of Jang Soo Kim.

\bigskip
\hrule

\noindent 2010 {\it Mathematics Subject Classification}:
Primary 05A05; Secondary 05A15, 15A15, 15B35, 11B83,  11C20.
\noindent \emph{Keywords:} Robbins numbers, plane partition, alternating sign matrix, Riordan array, Hankel determinant.

\bigskip
\hrule
\bigskip
\noindent (Concerned with sequences
\seqnum{A000045},
\seqnum{A000108},
\seqnum{A000124},
\seqnum{A000217},
\seqnum{A000975},
\seqnum{A001045},
\seqnum{A001764},
\seqnum{A001844},
\seqnum{A002061},
\seqnum{A002478},
\seqnum{A005021},
\seqnum{A005130},
\seqnum{A005156},
\seqnum{A005157},f
\seqnum{A005329},
\seqnum{A005448},
\seqnum{A005809},
\seqnum{A005891},
\seqnum{A006013},
\seqnum{A007226},
\seqnum{A007318},
\seqnum{A007440},
\seqnum{A030981},
\seqnum{A047098},
\seqnum{A047099},
\seqnum{A047749},
\seqnum{A049126},
\seqnum{A049130},
\seqnum{A052536},
\seqnum{A052547},
\seqnum{A052941},
\seqnum{A052975},
\seqnum{A072405},
\seqnum{A077954},
\seqnum{A077998},
\seqnum{A080937},
\seqnum{A080956},
\seqnum{A088927},
\seqnum{A094706},
\seqnum{A094832},
\seqnum{A094833},
\seqnum{A098746},
\seqnum{A099325},
\seqnum{A104769},
\seqnum{A106509},
\seqnum{A115140},
\seqnum{A120981},
\seqnum{A120984},
\seqnum{A121449},
\seqnum{A121545},
\seqnum{A122100},
\seqnum{A122368},
\seqnum{A127896},
\seqnum{A127897},
\seqnum{A130713},
\seqnum{A134565},
\seqnum{A154272},
\seqnum{A186185},
\seqnum{A188022},
\seqnum{A188687},
\seqnum{A200715},
\seqnum{A215404},
\seqnum{A305573},
\seqnum{A321511}, and
\seqnum{A339850}).

\end{document}